\documentclass[11pt]{article}
\usepackage{amsmath,amssymb,amsthm,indentfirst,enumitem,verbatim}
\usepackage[left=2.6cm,right=2.6cm,top=2.9cm,bottom=2.9cm]{geometry}
\newtheorem{theorem}{Theorem}[section]
\newtheorem{lemma}{Lemma}[section]
\newtheorem{definition}{Definition}[section]

\newtheorem{remark}{Remark}[section]
\numberwithin{equation}{section}
\allowdisplaybreaks
\begin{document}
	\title{Normalized solutions for $p$-Laplacian equations with potential}
	
	\author{Shengbing Deng\footnote{Corresponding author.}\ \footnote{E-mail address:\, {\tt shbdeng@swu.edu.cn} (S. Deng), {\tt qrwumath@163.com} (Q. Wu).}\  \ and Qiaoran Wu\\
	\footnotesize  School of Mathematics and Statistics, Southwest University, Chongqing, 400715, P.R. China}

    \date{}
	\maketitle
	
    \begin{abstract}
		In this paper, we consider the existence of normalized solutions for the following $p$-Laplacian equation
		    \begin{equation*}
		    	\left\{\begin{array}{ll}
		    		-\Delta_{p}u-V(x)\lvert u\rvert^{p-2}u+\lambda\lvert u\rvert^{p-2}u=\lvert u\rvert^{q-2}u&\mbox{in}\ \mathbb{R}^N,\\
		    		\int_{\mathbb{R}^N}\lvert u\rvert^pdx=a^p,
		    	\end{array}\right.
		    \end{equation*}
where $N\geqslant 1$, $p>1$, $p+\frac{p^2}{N}<q<p^*=\frac{Np}{N-p}$(if $N\leqslant p$, then $p^*=+\infty$), $a>0$ and $\lambda\in\mathbb{R}$ is a Lagrange multiplier which appears due to the mass constraint. Firstly, under some smallness assumptions on $V$, but no any assumptions on $a$, we obtain a mountain pass solution with positive energy, while no solution with negative energy. Secondly, assuming that the mass $a$ has an upper bound depending on $V$, we obtain two solutions, one is a local minimizer with negative energy, the other is a mountain pass solution with positive energy.
    \end{abstract}

\section{Introduction}

    In this paper, we study the existence of solutions for the following $p$-Laplacian equation
        \begin{equation}\label{equation}
        	-\Delta_{p}u-V(x)\lvert u\rvert^{p-2}u+\lambda\lvert u\rvert^{p-2}u=\lvert u\rvert^{q-2}u\quad\mbox{in}\ \mathbb{R}^N,
        \end{equation}
    where $\Delta_{p}u=\mbox{div}(\lvert\nabla u\rvert^{p-2}\nabla u)$ is the $p$-Laplacian operator, $N\geqslant 1$, $p>1$, $p+\frac{p^2}{N}<q<p^*=\frac{Np}{N-p}$(if $N\leqslant p$, then $p^*=+\infty$) and $\lambda\in\mathbb{R}$.

    If $p=2$, (\ref{equation}) is a special case of the following equation
        \begin{equation}\label{equationf}
        	-\Delta u-V(x)u+\lambda u=f(u)\quad\mbox{in}\ \mathbb{R}^N,
        \end{equation}
    which can be derived from the following Schr\"{o}dinger equation
        \[i\psi_{t}+\Delta\psi+V(x)\psi+g(\lvert\psi\rvert^2)\psi=0,\quad(t,x)\in\mathbb{R}^+\times\mathbb{R}^N,\]
    when we look for standing waves of the form $\psi(t,x)=e^{-i\lambda t}u(x)$, where $\lambda\in\mathbb{R}$, $u$ is a real function and $f(u)=g(u^2)u$. A lot of efforts have been done to study (\ref{equationf}). To obtain the existence result, a possible choice is to fix $\lambda\in\mathbb{R}$ and find critical points of functional
        \[I(u)=\frac{1}{2}\int_{\mathbb{R}^N}(\lvert\nabla u\rvert^2-V(x)u^2+\lambda u^2)dx-\int_{\mathbb{R}^N}F(u)dx,\]
    where $F(u)=\int_{0}^{u}f(s)ds$. Here we refer the readers to \cite{cg,cgpdss,immy} and references therein. Alternatively, an interesting method is to consider a prescribe $L^2$ norm of $u$, that is, let
        \begin{equation}\label{uL2}
        	\int_{\mathbb{R}^N}u^2dx=a^2,
        \end{equation}
    where $a>0$ is a fixed constant, and find critical points of functional
        \[J(u)=\frac{1}{2}\int_{\mathbb{R}^N}(\lvert\nabla u\rvert^2-V(x)u^2)dx-\int_{\mathbb{R}^N}F(u)dx,\]
    which satisfy (\ref{uL2}). In this case, $\lambda\in\mathbb{R}$ will appear as a Lagrange multiplier and solutions of (\ref{equationf}) satisfy (\ref{uL2}) always called normalized solutions.

    When we search for normalized solutions of (\ref{equationf}), a new critical exponent $2+\frac{4}{N}$ which is called $L^2$ critical exponent appears. This critical exponent can be derived from the Gagliardo-Nirenberg inequality (see \cite{am,wmi}): for every $p<q<p^*$, there exists an optimal constant $C_{N,p,q}>0$ depending on $N$, $p$ and $q$ such that
        \[\lVert u\rVert_{q}\leqslant C_{N,p,q}\lVert\nabla u\rVert_{p}^{\frac{N(q-p)}{pq}}\lVert u\rVert_{p}^{1-\frac{N(q-p)}{pq}}\quad\forall u\in W^{1,p}(\mathbb{R}^N).\]
    Assume that $V\equiv 0$ and $f(u)=\lvert u\rvert^{q-2}u$. In the $L^2$-subcritical case ( $2<q<2+\frac{4}{N}$), by Gagliardo-Nirenberg inequality, $J(u)$ is bounded from below if $u$ satisfies (\ref{uL2}) and we can find a global minimizer of $J(u)$, see \cite{acojc,lpl1,lpl2} and references therein.

    However, in the $L^2$-supercritical case ($q>2+\frac{4}{N}$), the functional $J(u)$ is unbounded from below if $u$ satisfies (\ref{uL2}) and it seems impossible to search for a global minimizer of $J(u)$. Jeanjean \cite{jl}   first studied the $L^2$-supercritical case, with general nonlinearities and $V\equiv 0$, he introduced an auxiliary functional
        \[\tilde{J}(s,u):=J(s\star u)=\frac{1}{2}e^{2s}\int_{\mathbb{R}^N}\lvert\nabla u\rvert^2dx-e^{-Ns}\int_{\mathbb{R}^N}F(e^{\frac{Ns}{2}}u)dx\]
    to obtain the boundedness of a Palais-Smale(PS) sequence and got the existence result. This method has been widely used to study normalized solutions of (\ref{equationf}). For the non-potential case, that is $V\equiv 0$, we refer the readers to \cite{jlltt,lx,sn1,sn2,wjwy} and references therein.

    If $V\not\equiv 0$, it is complicated to study normalized solutions of (\ref{equationf}). Firstly, the appearance of $V$ will strongly affect the geometry structure of $J$. In \cite{mrrgvg}, Molle et al. considered the normalized solutions of (\ref{equation}) with $V\geqslant 0$ and established different existence results under different assumptions on $V$ and $a$.
    Secondly, it is hard to obtain the compactness of a minimizing sequence or PS sequence. If $V$ is a radial function, we can work in radial space $H_{r}^1(\mathbb{R}^N)$ to overcome this difficulty, but it is invalid when $V$ is not a radial function. In \cite{btmrrm,mrrgvg}, the authors used a splitting lemma to conclude the weak convergence and obtained the compactness of PS sequence. In \cite{dyzx,kjctcl}, the authors considered a minimizing problem on Pohozaev manifold and proved the infimum on Pohozaev manifold is strictly decreasing for $a$ to obtain the compactness of a minimizing sequence. We refer the readers to \cite{btqszw,pbpavg,tzzczl,yjyj} and references therein for more results about normalized solutions of Schr\"{o}dinger equations with $V\not\equiv 0$.

    If $p\not\equiv 2$,  there are few papers on the normalized solutions of $p$-Laplacian equations. For the case of $V\equiv 0$, Wang et al. in \cite{wwlqzj} considered the equation
        \[-\Delta_{p}u+\lvert u\rvert^{p-2}u=\mu u+\lvert u\rvert^{s-1}u\quad\mbox{in}\ \mathbb{R}^N\]
    with $L^2$ constraint
        \[\int_{\mathbb{R}^N}u^2dx=a^2,\]
    where $1<p<N$, $s\in(\frac{N+2}{N}p,p^*)$. By Gaglairdo-Nirenberg inequality, the $L^2$-critical exponent should be $\frac{N+2}{N}p$. Zhang and Zhang \cite{zzzz} considered the $L^p$ constraint, that is
        \begin{equation*}
        	\left\{\begin{array}{ll}
        		-\Delta_{p}u=\lambda\lvert u\rvert^{p-2}u+\mu\lvert u\rvert^{q-2}u+g(u)&\mbox{in}\ \mathbb{R}^N,\\
        		\int_{\mathbb{R}^N}\lvert u\rvert^pdx=a^p,
        	\end{array}\right.
        \end{equation*}
    where $N> p$, $1<p<q\leqslant p+\frac{p^2}{N}$ and $g:\mathbb{R}\rightarrow\mathbb{R}$ are $L^2$-supercritical but Sobolev subcritical nonlinearities. By Gagliardo-Nirenberg inequality, the $L^p$-critical exponent is $p+\frac{p^2}{N}$. For the case of $V\not\equiv 0$, Wang and Sun \cite{wcsj} considered both the $L^2$ constraint and the $L^p$ constraint for the following problem
        \begin{equation*}
    	    \left\{\begin{array}{ll}
    		    -\Delta_{p}u+V(x)\lvert u\rvert^{p-2}u=\lambda\lvert u\rvert^{r-2}u+\lvert u\rvert^{q-2}u&\mbox{in}\ \mathbb{R}^N,\\
    		    \int_{\mathbb{R}^N}\lvert u\rvert^rdx=c,
    	    \end{array}\right.
        \end{equation*}
    where $1<p<N$, $\lambda\in\mathbb{R}$, $r=p$ or $2$, $p<q<p^*$ and $V(x)=\lvert x\rvert^k$ with $k>0$. Since $V$ is a radial function, they can work in the space $W_{r}^{1,p}(\mathbb{R}^N)$.

    Motivated by the results mentioned above, we are interested in the existence of normalized solutions for (\ref{equation}) with a $L^p$ constraint. Define $W(x):=V(x)\lvert x\rvert$ and
        \begin{equation}\label{constrain}
        	S_{a}:=\Big\{u\in W^{1,p}(\mathbb{R}^N): \int_{\mathbb{R}^N}\lvert u\rvert^pdx=a^p\Big\},
        \end{equation}
    where $a>0$ is a constant. Throughout the rest of the paper we assume that
        \begin{equation}\label{Vgeq0}
        	V\geqslant 0\quad\mbox{but}\quad V\not\equiv 0.
        \end{equation}
    Our results can be stated as follows.

    \begin{theorem}\label{th1}
        Let $N\geqslant 2$, $1<p<N$ and {\rm(\ref{Vgeq0})} holds. Then there exists a positive constant $L$ depending on $N$ and $p$, such that if
        	\begin{equation}\label{VNgeqp}
        	    \max\{\lVert V\rVert_{N/p},\lVert W\rVert_{N/(p-1)}\}<L,
        	\end{equation}
        then {\rm (\ref{equation})} has a mountain pass solution on $S_{a}$ for every $a>0$ with positive energy, while no solution with negative energy.
    \end{theorem}

    \begin{theorem}\label{th2}
        Let $N\geqslant 1$, $p>1$, $r\in(\max\{1,\frac{N}{p}\},+\infty]$, $s\in(\max\{\frac{p}{p-1},\frac{N}{p-1}\},+\infty]$ and {\rm(\ref{Vgeq0})} holds. Moreover, we assume that
        	\[V\in L^r(\mathbb{R}^N)\quad\mbox{and}\quad\lim_{\lvert x\rvert\rightarrow+\infty}V(x)=0\ \mbox{if}\ r=+\infty,\]
        and $W\in L^s(\mathbb{R}^N)$.
        	
        \noindent$(i)$
        \begin{minipage}[t]{\linewidth}
        	There exist positive constants $\sigma(N,p,q,r)$ and $L(N,p,q,r)$ such that if
        		\begin{equation}\label{Vcon2min}
        	    	a^{\sigma}\lVert V\rVert_{r}<K
        	    \end{equation}
       	    and there exists $\varphi\in S_{a}$:
       	        \begin{equation}\label{Vcon3min}
       	        	\int_{\mathbb{R}^N}(\lvert\nabla\varphi\rvert^p-V(x)\lvert\varphi\rvert^p)dx\leqslant 0,
       	        \end{equation}
            then {\rm(\ref{equation})} has a solution on $S_{a}$ which is a local minimizer with negative energy.
        \end{minipage}

        \noindent$(ii)$
        \begin{minipage}[t]{\linewidth}
            There exist positive constants $\sigma_{i}(N,p,q,r)$, $\bar{\sigma}_{i}(N,p,q,s)(i=1,2)$ and $L(N,p,q,r,s)$ such that if
                \begin{equation}\label{Vconmount}
            		\max\{a^{\sigma_{i}}\lVert V\rVert_{r},a^{\bar{\sigma}_{i}}\lVert W\rVert_{s}\}<L,\quad i=1,2,
            	\end{equation}
            then {\rm(\ref{equation})} has a mountain pass solution on $S_{a}$ with positive energy.
        \end{minipage}
    \end{theorem}
    \begin{remark}
    	We point out that if there exist $x_{0}\in\mathbb{R}^N$ and $\delta>0$ such that
    	    \begin{equation}\label{Vinf}
    	    	\eta:=\inf_{x\in B_{\delta}(x_{0})}V(x)>0\quad\mbox{with}\ \eta\delta^p>N\Big(\frac{N-p}{p-1}\Big)^{p-1}\ \mbox{if}\ N>p,
    	    \end{equation}
        then {\rm(\ref{Vcon3min})} holds, where $B_{\delta}(x_{0})$ denotes the open ball centered at $x_{0}$ with radius $\delta$. This fact will be proved in Lemma {\rm\ref{existphi}}.
    \end{remark}

    The paper is organized as follows. In Section 2, we collect some preliminary results. In Section 3, we give the proof of Theorem \ref{th1}. In Section 3 is devote to proving Theorem \ref{th2}.

\section{Preliminaries}
    In this section, we collect some useful results which will be used throughout the rest of paper. Define
        \[T(s):=\left\{\begin{array}{ll}
             s,&\mbox{if}\ \lvert s\rvert\leqslant 1,\\
            	\frac{s}{\lvert s\rvert},&\mbox{if}\ \lvert s\rvert>1.
            \end{array}\right.\]
    \begin{lemma}\label{nablad1p}
        Let $N\geqslant 1$, $p>1$ and $\{u_{n}\}\subset D^{1,p}(\mathbb{R}^N)$ such that $u_{n}\rightharpoonup u$ in $D^{1,p}(\mathbb{R}^N)$, where $D^{1,p}(\mathbb{R}^N)$ denotes the completion of $C_{c}^{\infty}(\mathbb{R}^N)$ with respect to the norm $\lVert u\rVert_{D^{1,p}}:=\lVert\nabla u\rVert_{p}$. Assume that for every $\varphi\in C_{c}^{\infty}(\mathbb{R}^N)$, there is
        	\begin{equation}\label{nablaT}
        	    \lim_{n\rightarrow\infty}\int_{\mathbb{R^N}}\varphi(\lvert\nabla u_{n}\rvert^{p-2}\nabla u_{n}-\lvert\nabla u\rvert^{p-2}\nabla u)\cdot\nabla T(u_{n}-u)dx=0.
        	\end{equation}
        Then, up to a subsequence, $\nabla u_{n}\rightarrow\nabla u$ a.e. in $\mathbb{R}^N$.
    \end{lemma}
        \begin{proof}
        	Let $k\in\mathbb{N}_{+}$ and $\varphi\in C_{c}^{\infty}(\mathbb{R}^N)$ satisfies
        	    \[0\leqslant\varphi\leqslant 1\quad\varphi_{k}=1\ \mbox{in}\ B_{k}\quad\mbox{and}\quad\varphi_{k}=0\ \mbox{in}\ \ B_{k+1}^c.\]
        	Since
        	    \[(\lvert\nabla u_{n}\rvert^{p-2}\nabla u_{n}-\lvert\nabla u\rvert^{p-2}\nabla u)\cdot\nabla T(u_{n}-u)\geqslant 0,\]
        	we have
        	    \begin{equation}
        	    	\lim_{n\rightarrow\infty}\int_{B_{k}}(\lvert\nabla u_{n}\rvert^{p-2}\nabla u_{n}-\lvert\nabla u\rvert^{p-2}\nabla u)\cdot\nabla T(u_{n}-u)dx=0.
        	    \end{equation}
        	Therefore, by \cite[Theorem 1.1]{dvswm}, up to a subsequence, we have $\nabla u_{k,n}\rightarrow\nabla u$ a.e. in $B_{k}$. Now, using the Cantor diagonal argument, we complete the proof.
        \end{proof}

        Let
            \[E_{\infty}(u)=\frac{1}{p}\lVert\nabla u\rVert_{p}^p-\frac{1}{q}\lVert u\rVert_{q}^q,\]

            \[E(u)=\frac{1}{p}\lVert\nabla u\rVert_{p}^p-\frac{1}{q}\lVert u\rVert_{q}^q-\frac{1}{p}\int_{\mathbb{R}^N}V(x)\lvert u\rvert^pdx,\]

            \[E_{\infty,\lambda}=\frac{1}{p}\lVert\nabla u\rVert_{p}^p+\frac{\lambda}{p}\lVert u\rVert_{p}^p-\frac{1}{q}\lVert u\rVert_{q}^q,\]
        and
            \[E_{\lambda}(u)=\frac{1}{p}\lVert\nabla u\rVert_{p}^p+\frac{\lambda}{p}\lVert u\rVert_{p}^p-\frac{1}{q}\lVert u\rVert_{q}^q-\frac{1}{p}\int_{\mathbb{R}^N}V(x)\lvert u\rvert^pdx.\]
        \begin{lemma}\label{convergence}
        	Let $N\geqslant 1$, $p>1$ and
        	
        	 $(i)$ If $N>p$, then $V\in L^{N/p}(B_{1})$ and $V\in L^{\tilde{r}}(\mathbb{R}^N\backslash B_{1})$ for some $\tilde{r}\in[N/p,+\infty]$,
        	
        	 $(ii)$ If $N\leqslant p$, then $V\in L^{r}(B_{1})$ and $V\in L^{\tilde{r}}(\mathbb{R}^N\backslash B_{1})$ for some $r,\tilde{r}\in (1,+\infty]$.
        	
        	\noindent If $\{u_{n}\}$ is a PS sequence for $E_{\lambda}$ in $W^{1,p}(\mathbb{R}^N)$ and $u_{n}\rightharpoonup u$ in $W^{1,p}(\mathbb{R}^N)$. Then, $\nabla u_{n}\rightarrow\nabla u$ a.e. in $\mathbb{R}^N$.
        \end{lemma}
        \begin{proof}
        	By Lemma \ref{nablad1p}, we just need to prove (\ref{nablaT}). Since $u_{n}\rightharpoonup u$ in $W^{1,p}(\mathbb{R}^N)$, we can assume that $u_{n}\rightarrow u$ a.e. in $\mathbb{R}^N$. Therefore, the Egorov's theorem implies that for every $\delta>0$, there exists $F_{\delta}\subset\mbox{supp}\varphi$ such that $u_{n}\rightarrow u$ uniformly in $F_{\delta}$ and $m(\mbox{supp}\varphi\backslash F_{\delta})<\delta$. Hence, $\lvert u_{n}(x)-u(x)\rvert\leqslant 1$ for all $x\in F_{\delta}$ as long as $n$ sufficiently large.
        	
        	Now, we have
        	    \begin{align*}
        	    	&\limsup_{n\rightarrow\infty}\Big|\int_{\mathbb{R}^N}\varphi\lvert\nabla u\rvert^{p-2}\nabla u\cdot\nabla T(u_{n}-u)dx\Big|\nonumber\\
        	    	\leqslant&\limsup_{n\rightarrow\infty}\Big|\int_{F_{\delta}}\varphi\lvert\nabla u\rvert^{p-2}\nabla u\cdot\nabla T(u_{n}-u)dx\Big|+\limsup_{n\rightarrow\infty}\Big|\int_{F_{\delta}^c}\varphi\lvert\nabla u\rvert^{p-2}\nabla u\cdot\nabla T(u_{n}-u)dx\Big|\nonumber\\
        	    	=&\limsup_{n\rightarrow\infty}\Big|\int_{F_{\delta}^c}\varphi\lvert\nabla u\rvert^{p-2}\nabla u\cdot\nabla T(u_{n}-u)dx\Big|,
        	    \end{align*}
            since $u_{n}\rightharpoonup u$ in $W^{1,p}(\mathbb{R}^N)$. For every $\varepsilon>0$, by the definition of $T$,
                \[\Big|\int_{F_{\delta}^c}\varphi\lvert\nabla u\rvert^{p-2}\nabla u\cdot\nabla T(u_{n}-u)dx\Big|\leqslant\int_{F_{\delta}^c}\lvert\varphi\rvert\lvert\nabla u\rvert^{p-1}dx<\varepsilon,\]
            as long as $\delta$ sufficiently small, which implies
                \begin{equation}\label{nablau}
                	\lim_{n\rightarrow\infty}\int_{\mathbb{R}^N}\varphi\lvert\nabla u\rvert^{p-2}\nabla u\cdot\nabla T(u_{n}-u)dx=0.
                \end{equation}

            Next, we prove
                \[\lim_{n\rightarrow\infty}\int_{\mathbb{R^N}}\varphi\lvert\nabla u_{n}\rvert^{p-2}\nabla u_{n}\cdot\nabla T(u_{n}-u)dx=0,\]
            which together with (\ref{nablau}) implies (\ref{nablaT}) holds. Since $\{u_{n}\}$ is a PS sequence for $E_{\lambda}$ in $W^{1,p}(\mathbb{R}^N)$, we have
                \begin{align*}
                	\int_{\mathbb{R}^N}\lvert\nabla u_{n}\rvert^{p-2}\nabla u_{n}\cdot\nabla\psi dx=&-\lambda\int_{\mathbb{R}^N}\lvert u_{n}\rvert^{p-2}u_{n}\psi dx+\int_{\mathbb{R}^N}V(x)\lvert u_{n}\rvert^{p-2}u_{n}\psi dx\\
                	&+\int_{\mathbb{R^N}}\lvert u_{n}\rvert^{q-2}u_{n}\psi dx+o_{n}(1)\lVert\psi\rVert_{W^{1,p}},
                \end{align*}
            for every $\psi\in W^{1,p}(\mathbb{R}^N)$. Let $\psi=\varphi T(u_{n}-u)$, then
                \begin{align}\label{nablaun}
                	&\limsup_{n\rightarrow\infty}\Big\lvert\int_{\mathbb{R^N}}\varphi\lvert\nabla u_{n}\rvert^{p-2}\nabla u_{n}\cdot\nabla T(u_{n}-u)dx\Big\rvert\nonumber\\
                	\leqslant&\limsup_{n\rightarrow\infty}\int_{\mathbb{R}^N}\Big(\lvert\nabla u_{n}\rvert^{p-1}\lvert T(u_{n}-u)\nabla\varphi\rvert+\lvert\lambda\rvert\lvert u_{n}\rvert^{p-1}\lvert\varphi T(u_{n}-u)\rvert\nonumber\\
                	&+\lvert V(x)\rvert\lvert u_{n}\rvert^{p-1}\lvert\varphi T(u_{n}-u)\rvert+\lvert u_{n}\rvert^{q-1}\lvert\varphi T(u_{n}-u)\rvert\Big)dx.
                \end{align}
            We know
                \begin{equation*}
                	\limsup_{n\rightarrow\infty}\int_{\mathbb{R}^N}\lvert V(x)\rvert\lvert u_{n}\rvert^{p-1}\lvert\varphi T(u_{n}-u)\rvert dx\leqslant\limsup_{n\rightarrow\infty}\int_{F_{\delta}^c}\lvert V(x)\rvert\lvert u_{n}\rvert^{p-1}\lvert\varphi\rvert dx.
                \end{equation*}
            If $N>p$, by the H\"older inequality
                \begin{equation*}
                	\limsup_{n\rightarrow\infty}\int_{F_{\delta}^c\cap B_{1}}\lvert V(x)\rvert\lvert u_{n}\rvert^{p-1}\lvert\varphi\rvert dx\leqslant\limsup_{n\rightarrow\infty}\lVert V\rVert_{L^{\frac{N}{p}}(B_{1})}\lVert u_{n}\rVert_{p^*}^{p-1}\lVert\varphi\rVert_{L^{\frac{Np}{Np+p-N}}(F_{\delta}^c)}<\varepsilon,
                \end{equation*}
            and
                \begin{align*}
                	\limsup_{n\rightarrow\infty}\int_{F_{\delta}^c\cap B_{1}^c}\lvert V(x)\rvert\lvert u_{n}\rvert^{p-1}\lvert\varphi\rvert dx
                	\leqslant \limsup_{n\rightarrow\infty}\lVert V\rVert_{L^{\tilde{r}}(B_{1}^c)}\lVert u_{n}\rVert_{p^*}^{p-1}\lVert\varphi\rVert_{L^{\frac{Np\tilde{r}}{Np\tilde{r}-r+(N-p)(p-1)}}(F_{\delta}^c)}<\varepsilon,
                \end{align*}
            for every $\varepsilon>0$, by taking $\delta$ sufficiently small. Therefore, we can ensure that
                \[\limsup_{n\rightarrow\infty}\int_{\mathbb{R}^N}\lvert V(x)\rvert\lvert u_{n}\rvert^{p-1}\lvert\varphi T(u_{n}-u)\rvert dx\leqslant C\varepsilon.\]
            If $N\leqslant p$, we can also guarantee that
               \[\limsup_{n\rightarrow\infty}\int_{\mathbb{R}^N}\lvert V(x)\rvert\lvert u_{n}\rvert^{p-1}\lvert\varphi T(u_{n}-u)\rvert dx\leqslant C\varepsilon.\]
            Similarly, we have
                \[\limsup_{n\rightarrow\infty}\int_{\mathbb{R}^N}\lvert\nabla u_{n}\rvert^{p-1}\lvert T(u_{n}-u)\nabla\varphi\rvert dx\leqslant C\varepsilon,\]
                \[\limsup_{n\rightarrow\infty}\int_{\mathbb{R}^N}\lvert u_{n}\rvert^{p-1}\lvert\varphi T(u_{n}-u)\rvert dx\leqslant C\varepsilon,\]
            and
                \[\limsup_{n\rightarrow\infty}\int_{\mathbb{R}^N}\lvert u_{n}\rvert^{q-1}\lvert\varphi T(u_{n}-u)\rvert dx\leqslant C\varepsilon.\]
            Therefore, from (\ref{nablaun}), we obtain
                \[\limsup_{n\rightarrow\infty}\Big\lvert\int_{\mathbb{R}^N}\varphi\lvert\nabla u_{n}\rvert^{p-2}\nabla u_{n}\cdot\nabla T(u_{n}-u)dx\Big\rvert\leqslant C\varepsilon,\]
            which implies
                \[\lim_{n\rightarrow\infty}\int_{\mathbb{R^N}}\varphi\lvert\nabla u_{n}\rvert^{p-2}\nabla u_{n}\cdot\nabla T(u_{n}-u)dx=0.\]
        \end{proof}

        \begin{remark}\label{usolution}
        	If $\{u_{n}\}$ is a PS sequence for $E_{\lambda}$ in $W^{1,p}(\mathbb{R}^N)$ and $u_{n}\rightharpoonup u$ in $W^{1,p}(\mathbb{R}^N)$. Then, by Lemma {\rm\ref{convergence}} and weak convergence, $u$ is a solution of {\rm(\ref{equation})}.
        \end{remark}

        \begin{lemma}\label{vnps}
        	Let $N\geqslant 1$, $p>1$, and $V$ satisfies the assumptions of Lemma {\rm \ref{convergence}}. Assume $\{u_{n}\}$ is a PS sequence for $E_{\lambda}$ in $W^{1,p}(\mathbb{R}^N)$, and $u_{n}\rightarrow u$ in $W^{1,p}(\mathbb{R}^N)$. Let $v_{n}=u_{n}-u$. Then, $\{v_{n}\}$ is a PS sequence for $E_{\infty,\lambda}$.
        \end{lemma}
        \begin{proof}
        	Since $u_{n}\rightharpoonup u$ in $W^{1,p}(\mathbb{R}^N)$, we have $v_{n}\rightharpoonup 0$ in $W^{1,p}(\mathbb{R}^N)$, $v_{n}\rightarrow 0$ in $L_{loc}^{p}(\mathbb{R}^N)$, $L_{loc}^q(\mathbb{R}^N)$, and a.e. in $\mathbb{R}^N$. Set
        	    \begin{align}\label{V}
        		    \int_{\mathbb{R}^N}V(x)\lvert v_{n}\rvert^pdx=\int_{B_{1}}V(x)\lvert v_{n}\rvert^pdx+\int_{B_{1}^c}V(x)\lvert v_{n}\rvert^pdx=A_{n}+B_{n}.
        	    \end{align}
        	Firstly, we assume $N>p$ and $\tilde{r}<+\infty$. Let $\tilde{r}'$ be the conjugate exponent of $\tilde{r}$. Since $\{\lvert v_{n}\rvert^p\}$ is bounded in $L^{N/(N-p)}(B_{1})$ and $L^{\tilde{r}'}(B_{1}^c)$, we have $\lvert v_{n}\rvert^p\rightharpoonup 0$ in $L^{N/(N-p)}(B_{1})$ and $L^{\tilde{r}'}(B_{1}^c)$ and hence $A_{n},B_{n}\rightarrow 0$ as $n\rightarrow\infty$. In the case $\tilde{r}=+\infty$, it is not difficult to prove that $A_{n}\rightarrow 0$ as $n\rightarrow\infty$. We know $v_{n}\rightarrow 0$ in $L_{loc}^p(\mathbb{R}^N)$, thus, for every $R>1$, there holds
        	    \[\limsup_{n\rightarrow\infty}\lvert B_{n}\rvert=\limsup_{n\rightarrow\infty}\Big\lvert\int_{B_{R}^c}V(x)\lvert v_{n}\rvert^{p}dx\Big\rvert\leqslant C\sup_{B_{R}^c}\lvert V\rvert,\]
        	which implies $B_{n}\rightarrow 0$ as $n\rightarrow\infty$. Similar arguments for $N\leqslant p$ can also prove that $A_{n},B_{n}\rightarrow 0$ as $n\rightarrow\infty$. To sum up, in any case we have
        	    \begin{equation}\label{Vcon0}
        		    \int_{\mathbb{R}^N}V(x)\lvert v_{n}\rvert^pdx\rightarrow 0\quad\mbox{as}\ n\rightarrow\infty,
        	    \end{equation}
        	which implies
        	    \begin{equation}\label{VunptendVup}
        	    	\int_{\mathbb{R}^N}V(x)\lvert u_{n}\rvert^pdx\rightarrow\int_{\mathbb{R}^N}V(x)\lvert u\rvert^pdx\quad\mbox{as}\ n\rightarrow\infty.
        	    \end{equation}
        	
        	Since $\{u_{n}\}$ is a PS sequence for $E_{\lambda}$, there exists $c\in\mathbb{R}$ such that
        	    \[E_{\lambda}(u_{n})\rightarrow c\quad\mbox{and}\quad\lVert E_{\lambda}'(u_{n})\rVert\rightarrow 0\ \mbox{in}\ W^{-1,p}(\mathbb{R}^N)\quad\mbox{as}\ n\rightarrow\infty.\]
        	By Br\'{e}zis-Lieb lemma and Lemma \ref{convergence}, we have
        	     \[E_{\lambda}(u_{n})=E_{\lambda}(u)+E_{\infty,\lambda}(v_{n})+o_{n}(1),\]
        	which implies
        	    \begin{equation}\label{Einflamvn}
        	    	E_{\infty,\lambda}(v_{n})\rightarrow c-E_{\lambda}(u)\quad\mbox{as}\ n\rightarrow\infty.
        	    \end{equation}
        	
        	Finally, we prove that
        	    \begin{equation*}
        	    	\lVert E_{\infty,\lambda}'(v_{n})\rVert_{W^{-1,p}}\rightarrow 0\quad\mbox{as}\ n\rightarrow\infty,
        	    \end{equation*}
            which together with (\ref{Einflamvn}) implies $\{v_{n}\}$ is a PS sequence for $E_{\infty,\lambda}$. We just need to prove that
                \[E_{\infty,\lambda}'(v_{n})\psi=o_{n}(1)\lVert\psi\rVert_{W^{1,p}}\quad\forall\psi\in W^{1,p}(\mathbb{R}^N),\]
        	that is
        	    \begin{equation}\label{Einflamdervn}
        	    	\int_{\mathbb{R}^N}(\lvert\nabla v_{n}\rvert^{p-2}\nabla v_{n}\cdot\nabla\psi+\lambda\lvert v_{n}\rvert^{p-2}v_{n}\psi-\lvert v_{n}\rvert^{q-2}v_{n}\psi)dx=o_{n}(1)\lVert\psi\rVert_{W^{1,p}}.
        	    \end{equation}
            By the H\"{o}lder inequality,
                \begin{align*}
                	&\Big|\int_{\mathbb{R}^N}(\lvert\nabla u_{n}\rvert^{p-2}\nabla u_{n}-\lvert\nabla v_{n}\rvert^{p-2}\nabla v_{n}-\lvert\nabla u\rvert^{p-2}\nabla u)\cdot\nabla\psi dx\Big|\\
                	\leqslant&\Big(\int_{\mathbb{R}^N}\big|\lvert\nabla u_{n}\rvert^{p-2}\nabla u_{n}-\lvert v_{n}\rvert^{p-2}\nabla v_{n}-\lvert\nabla u\rvert^{p-2}\nabla u\big|^{\frac{p}{p-1}}dx\Big)^{\frac{p}{p-1}}\lVert\nabla\psi\rVert_{p}\\
                	\leqslant&\Big(\int_{\mathbb{R}^N}\big|\lvert\nabla u_{n}\rvert^{p-2}\nabla u_{n}-\lvert v_{n}\rvert^{p-2}\nabla v_{n}-\lvert\nabla u\rvert^{p-2}\nabla u\big|^{\frac{p}{p-1}}dx\Big)^{\frac{p}{p-1}}\lVert\psi\rVert_{W^{1,p}}.
                \end{align*}
            From \cite[Lemma 3.2]{mcwm}, we know
                \[\int_{\mathbb{R}^N}\big|\lvert\nabla u_{n}\rvert^{p-2}\nabla u_{n}-\lvert v_{n}\rvert^{p-2}\nabla v_{n}-\lvert\nabla u\rvert^{p-2}\nabla u\big|^{\frac{p}{p-1}}dx=o_{n}(1),\]
            which implies
                \begin{align}\label{BLnabla}
                	&\int_{\mathbb{R}^N}\lvert\nabla v_{n}\rvert^{p-2}\nabla v_{n}\cdot\nabla\psi dx\nonumber\\
                	=&\int_{\mathbb{R}^N}\lvert\nabla u_{n}\rvert^{p-2}\nabla u_{n}\cdot\nabla\psi dx-\int_{\mathbb{R}^N}\lvert\nabla u\rvert^{p-2}\nabla u\cdot\nabla\psi dx+o_{n}(1)\lVert\psi\rVert_{W^{1,p}}.
                \end{align}
            Similarly, we have
                \begin{equation}\label{BLp}
                	\int_{\mathbb{R}^N}\lvert v_{n}\rvert^{p-2}v_{n}\psi dx=\int_{\mathbb{R}^N}\lvert u_{n}\rvert^{p-2}u_{n}\psi dx-\int_{\mathbb{R}^N}\lvert u\rvert^{p-2}u\psi dx+o_{n}(1)\lVert\psi\rVert_{W^{1,p}}.
                \end{equation}
            and
                \begin{equation}\label{BLq}
                	\int_{\mathbb{R}^N}\lvert v_{n}\rvert^{q-2}v_{n}\psi dx=\int_{\mathbb{R}^N}\lvert u_{n}\rvert^{q-2}u_{n}\psi dx-\int_{\mathbb{R}^N}\lvert u\rvert^{q-2}u\psi dx+o_{n}(1)\lVert\psi\rVert_{W^{1,p}}.
                \end{equation}
            Now, using (\ref{VunptendVup}), (\ref{BLnabla}), (\ref{BLp}), (\ref{BLq}), Remark \ref{usolution} and the fact that $\{u_{n}\}$ is a PS sequence for $E_{\lambda}$, we obtain (\ref{Einflamdervn}).
        \end{proof}

        Now, we state a splitting Lemma for $E_{\lambda}$. The idea of this proof comes from \cite[Lemma 3.1]{bvcg} and \cite[Lemma 2.3]{mrrgvg}.

        \begin{lemma}\label{splittig}
        	Let $N\geqslant 1$, $p>1$, and assume that
        	
        	 $(i)$ $N>p$: $V\in L^{N/p}(B_{1})$, and $V\in L^{\tilde{r}}(\mathbb{R}^N\backslash B_{1})$ for some $\tilde{r}\in[N/p,+\infty]$,
        	
        	 $(ii)$ $N\leqslant p$: $V\in L^{r}(B_{1})$ and $V\in L^{\tilde{r}}(\mathbb{R}^N\backslash B_{1})$ for some $r,\tilde{r}\in (1,+\infty]$,
        	
        	 $(iii)$ in case $\tilde{r}=+\infty$, $V$ further satisfies $V(x)\rightarrow 0$ as $\lvert x\rvert\rightarrow+\infty$,
        	
        	 $(iv)$ $\lambda>0$.\\
        If $\{u_{n}\}$ is a PS sequence for $E_{\lambda}$ in $W^{1,p}(\mathbb{R}^N)$ and $u_{n}\rightharpoonup u$ in $W^{1,p}(\mathbb{R}^N)$ but not strongly, then there exist an integer $k\geqslant 1$, $k$ nontrival solutions $w^1,...,w^k\in W^{1,p}(\mathbb{R}^N)$ to the equation
        	    \begin{equation}\label{limequ}
        	    	-\Delta w+\lambda w=\lvert w\rvert^{q-2}w,
        	    \end{equation}
            and $k$ sequence $\{y_{n}^j\}\subset\mathbb{R}^N$, $1\leqslant j\leqslant k$, such that $\lvert y_{n}^j\rvert\rightarrow+\infty$ as $n\rightarrow\infty$, $\lvert y_{n}^{j_{1}}-y_{n}^{j_{2}}\rvert\rightarrow+\infty$ for $j_{1}\neq j_{2}$ as $n\rightarrow\infty$, and
                \begin{equation}\label{decom}
                	u_{n}=u+\sum_{j=1}^{k}w^j(\cdot-y_{n}^j)+o_{n}(1)\quad\mbox{in}\ W^{1,p}(\mathbb{R}^N).
                \end{equation}
            Moreover, we have
                \begin{equation}\label{decomnorm}
                	\lVert u_{n}\rVert_{p}^p=\lVert u\rVert_{p}^p+\sum_{j=1}^{k}\lVert w^j\rVert_{p}^p+o_{n}(1),
                \end{equation}
            and
                \begin{equation}\label{decomener}
                	E_{\lambda}(u_{n})=E_{\lambda}(u)+\sum_{j=1}^{k}E_{\infty,\lambda}(w^j)+o_{n}(1).
                \end{equation}
        \end{lemma}
        \begin{proof}
        	Let $u_{1,n}=u_{n}-u$. Then $u_{1,n}\rightharpoonup 0$ in $W^{1,p}(\mathbb{R}^N)$, $u_{n}\rightarrow u$ in $L_{loc}^{p}(\mathbb{R}^N)$, $L_{loc}^q(\mathbb{R}^N)$ and a.e. in $\mathbb{R}^N$.
        	Similar to the proof of Lemma \ref{vnps}, we can prove that
                \begin{equation}\label{Vu1npcon0}
                	\int_{\mathbb{R}^N}V(x)\lvert u_{1,n}\rvert^pdx\rightarrow 0,\quad\mbox{as}\ n\rightarrow\infty.
                \end{equation}
            Since $u_{n}\rightharpoonup u$ in $W^{1,p}(\mathbb{R}^N)$ but not strongly, there is
                \[\liminf_{n\rightarrow\infty}\lVert u_{1,n}\rVert^p>0.\]
            By Lemma \ref{vnps}, we know $\{u_{1,n}\}$ is a PS sequence for $E_{\infty,\lambda}$, thus
                \[\lVert\nabla u_{1,n}\rVert_{p}^p+\lambda\lVert u_{1,n}\rVert_{p}^p=\lVert u_{1,n}\rVert_{q}^q+o_{n}(1),\]
            which implies
                \[\liminf_{n\rightarrow\infty}\lVert u_{1,n}\rVert_{q}^q>0,\]
            since $\lambda>0$.

            Let us decompose $\mathbb{R}^N$ into $N$-dimensional unit hypercubes $Q_{i}$ and set
                \[l_{n}=\sup_{i\in\mathbb{N}_{+}}\lVert u_{1,n}\rVert_{L^q(Q_{i})}.\]
            It is not difficult to conclude that $l_{n}$ can be attained by some $i_{n}\in\mathbb{N}_{+}$. We claim that
                \[\liminf_{n\rightarrow\infty}l_{n}>0.\]
            Indeed,
                \begin{align*}
                	\lVert u_{1,n}\rVert_{q}^q=\sum_{i=1}^{\infty}\lVert u_{1,n}\rVert_{L^q(Q_{i})}^q
                	\leqslant l_{n}^{q-p}\sum_{i=1}^{\infty}\lVert u_{1,n}\rVert_{L^q(Q_{i})}^p\leqslant Cl_{n}^{q-p}\sum_{i=1}^{n}\lVert\nabla u_{1,n}\rVert_{L^p(Q_{i})}^p\leqslant Cl_{n}^{q-p},
                \end{align*}
            this prove the claim.

            Let $y_{n}^1$ be the center of $Q_{i_{n}}$. It is not difficult to observe that $\lvert y_{n}^1\rvert\rightarrow+\infty$, since $u_{1,n}\rightarrow 0$ in $L_{loc}^q(\mathbb{R}^N)$. Set
                \[v_{1,n}:=u_{1,n}(\cdot+y_{n}^1),\]
            then $\{v_{1,n}\}$ is a PS sequence for $E_{\infty,\lambda}$ and there exists $w^1\in W^{1,p}(\mathbb{R}^N)\backslash\{0\}$ such that $v_{1,n}\rightharpoonup w^1$ in $W^{1,p}(\mathbb{R}^N)$. By weak convergence, we know $w^1$ satisfies (\ref{limequ}). Moreover, by (\ref{Vcon0}), Br\'{e}zis-Lieb Lemma\cite[Theorem 2]{bhle} and Lemma \ref{convergence}, we have
                \[u_{n}=u+u_{1,n}=u+v_{1,n}(\cdot-y_{n}^1)=u+w^1(\cdot-y_{n}^1)+[v_{1,n}(\cdot-y_{n}^1)-w^1(\cdot-y_{n}^1)],\]
                \[\lVert u_{n}\rVert^p=\lVert u\rVert^p+\lVert w^1\rVert^p+\lVert v_{1,n}-w^1\rVert^p+o_{n}(1),\]
            and
                \[E_{\lambda}(u_{n})=E_{\lambda}(u)+F_{\infty,\lambda}(w^1)+F_{\infty,\lambda}(v_{1,n}-w^1)+o_{n}(1).\]

            Now, we can set
                \[u_{2,n}=v_{1,n}(\cdot-y_{n}^1)-w^1(\cdot-y_{n}^1),\]
            and iterate the above procedure. To complete the proof, we just need to prove that the iteration will be ended in finite steps. Suppose by contraction that the iteration will not be ended, then we have
                \[\lVert u_{n}\rVert^p\geqslant\sum_{j=1}^{\infty}\lVert w^j\rVert^p.\]
            In Lemma \ref{w}, we will claim that there exists a constant $C>0$ depending on $N$, $p$, $q$ and $\lambda$, such that $\lVert w^j\rVert_{W^{1,p}}\geqslant C$ which implies $\lVert u_{n}\rVert=+\infty$, which is an absurd.
        \end{proof}

        Finally, we give some properties of $w$ which satisfies
            \[-\Delta_{p}w+\lambda\lvert w\rvert^{p-2}w=\lvert w\rvert^{q-2}w\]
        for some $\lambda>0$. Define
           \begin{equation}\label{Z}
           	    Z_{a}:=\{w\in S_{a}:\exists\lambda>0,\mbox{s.t.}-\Delta_{p}w+\lambda\lvert w\rvert^{p-2}w=\lvert w\rvert^{q-2}w\},
           \end{equation}
        and
            \begin{equation}\label{ma}
            	m_{a}:=\inf_{w\in Z_{a}}E_{\infty}(w).
            \end{equation}

                \begin{lemma}\label{w}
        	Let $w\in W^{1,p}(\mathbb{R}^N)$ be a non-trivial solution of
        	\[-\Delta_{p}w+\lambda\lvert w\rvert^{p-2}w=\lvert w\rvert^{q-2}w\]
        	for some $\lambda>0$. Then there exists a constant $C$ depending on $N$, $p$, $q$, and $\lambda$ such that
        	\[\lVert w\rVert_{W^{1,p}}\geqslant C.\]
        \end{lemma}
        \begin{proof}
        	By the Pohozaev identity, we have
        	\[\lVert\nabla w\rVert_{p}^p=\frac{N(q-p)}{pq}\lVert w\rVert_{q}^q,\]
        	which together with
        	\[\lVert\nabla w\rVert_{p}^p+\lambda\lVert w\rVert_{p}^p=\lVert w\rVert_{q}^q\]
        	implies
        	\[\lambda\lVert w\rVert_{p}^p=\frac{Np-(N-p)q}{N(q-p)}\lVert\nabla w\rVert_{p}^p.\]
        	Now, by Gagliardo-Nirenberg inequality, we have
        	\[\frac{N(q-p)}{pq}\lVert\nabla w\rVert_{p}^p=\lVert w\rVert_{q}^q\leqslant C_{N,p,q}^q\lVert\nabla w\rVert_{p}^{\frac{N(q-p)}{p}}\lVert w\rVert_{p}^{q-\frac{N(q-p)}{p}}=C\lVert\nabla w\rVert_{p}^{q}.\]
        	Therefore, $\lVert w\rVert_{W^{1,p}}\geqslant C=C(N,p,q,\lambda)$, since $q>p$.
        \end{proof}

        \begin{lemma}\label{mageqz}
        	We have $m_{a}>0$ and $m_{a}$ can be achieved by some $w\in Z_{a}$.
        \end{lemma}
        \begin{proof}
        	For every $w\in Z_{a}$, by the Pohozaev identity and Gagliardo-Nirenberg inequality, we have
        	    \begin{equation}\label{GNw}
        	    	\lVert\nabla w\rVert_{p}^p=\frac{N(q-p)}{pq}\lVert w\rVert_{q}^q\leqslant\frac{N(q-p)}{pq}C_{N,p,q}^qa^{q-\frac{N(q-p)}{p}}\lVert\nabla w\rVert_{p}^{\frac{N(q-p)}{p}},
        	    \end{equation}
            which implies $\lVert\nabla w\rVert_{p}\geqslant C=C(N,p,q,a)$, since $N(q-p)/p>p$. Therefore,
                \[E_{\infty}(w)=\frac{1}{p}\lVert\nabla w\rVert_{p}^p-\frac{1}{q}\lVert w\rVert_{q}^q=\Big(\frac{1}{p}-\frac{1}{q\gamma_{q}}\Big)\lVert\nabla w\rVert_{p}^p\geqslant C,\]
            which implies $m_{a}>0$. By \cite[Page 2]{am}, it is not difficult to know the inequality in (\ref{GNw}) can become an equation by some $\psi\in Z_{a}$ and $E_{\infty}(\psi)=m_{a}$.
        \end{proof}

        \begin{lemma}\label{made}
        	$m_{a}$ is decreasing for $a\in(0,+\infty)$ and
        	    \[a^{p\theta}m_{a}=b^{p\theta}m_{b}\quad\forall a,b>0,\]
        	where
        	    \[\theta=\frac{Np-q(N-p)}{N(q-p)-p^2}.\]
        \end{lemma}
        \begin{proof}
        	By Lemma \ref{mageqz}, there exists $w_{b}\in Z_{b}$ such that $E_{\infty}(w_{b})=m_{b}$. Moreover, there exists $\lambda_{b}>0$ such that
        	    \[-\Delta_{p}w_{b}+\lambda_{b}\lvert w_{b}\rvert^{p-2}w_{b}=\lvert w_{b}\rvert^{q-2}w_{b}.\]
        	Therefore, by Pohozaev identity, we have
        	    \begin{equation}\label{wbmb}
        	    	E_{\infty}(w_{b})=\frac{1}{p}\lVert\nabla w_{b}\rVert_{p}^p-\frac{1}{q}\lVert w_{b}\rVert_{q}^q=\frac{N(q-p)-p^2}{Np(q-p)}\lVert\nabla w_{b}\rVert_{p}^p=m_{b}.
        	    \end{equation}

            Let $w=\alpha w_{b}(\beta\cdot)$, where
                \[\alpha=\Big(\frac{b}{a}\Big)^{\frac{p^2}{N(q-p)-p^2}}\quad\mbox{and}\quad\beta=\Big(\frac{b}{a}\Big)^{\frac{p(q-p)}{N(q-p)-p^2}}.\]
            Direct calculations show that $w\in Z_{a}$ and
                \[-\Delta_{p}w+\lambda_{b}\beta^p\lvert w\rvert^{p-2}w=\lvert w\rvert^{q-2}w.\]
            Thus, $E_{\infty}(w)\geqslant m_{a}$. By (\ref{wbmb}), we have
                \begin{align*}
                	E_{\infty}(w)&=\frac{1}{p}\lVert\nabla w\rVert_{p}^p-\frac{1}{q}\lVert w\rVert_{q}^q=\frac{N(q-p)-p^2}{Np(q-p)}\lVert\nabla w\rVert_{p}^p\\
                	&=\frac{N(q-p)-p^2}{Np(q-p)}\alpha^p\beta^{p-N}\lVert\nabla w_{b}\rVert_{p}^p=\Big(\frac{b}{a}\Big)^{p\theta}m_{b},
                \end{align*}
            which implies
                \[b^{p\theta}m_{b}\geqslant a^{p\theta}m_{a}.\]
            Similarly, we have
                \[a^{p\theta}m_{a}\geqslant b^{p\theta}m_{b}.\]
        \end{proof}

\section{Proof of Theorem \ref{th1}}
    In this section, we will prove Theorem \ref{th1}. For some fixed $\delta\in(0,1)$, we give following assumptions on $V$ and $W$.
        \begin{equation}\label{Vcon1}
        	\lVert V\rVert_{N/p}\leqslant(1-\delta)S,
        \end{equation}
        \begin{equation}\label{Vcon2}
        	N\lvert q-2p\rvert S^{-1}\lVert V\rVert_{N/p}+p^2S^{-\frac{p-1}{p}}\lVert W\rVert_{N/(p-1)}<B,
        \end{equation}
        \begin{equation}\label{Vcon3}
        	\big[AN\lvert q-2p\rvert+(N-p)D\big]S^{-1}\lVert V\rVert_{N/p}+p(pA+D)S^{-\frac{p-1}{p}}\lVert W\rVert_{N/(p-1)}<AB,
        \end{equation}
    and
        \begin{equation}\label{Vcon4}
        	N\lvert q-2p\rvert S^{-1}\lVert V\rVert_{N/p}+p^2S^{-\frac{p-1}{p}}\lVert W\rVert_{N/(p-1)}\leqslant N(q-p)-p^2
        \end{equation}
    where
        \[A=Np-(N-p)q,\quad B=N(q-p)-p^2\quad\mbox{and}\quad D=N(q-p)^2.\]
    Obviously, under the above assumptions, (\ref{VNgeqp}) can be established. Firstly, we prove that under the above assumptions, the functional $E$ has a mountain pass geometry.

    \begin{lemma}\label{Fgeq}
    	For every $u\in S_{a}$,
    	    \[E(u)\geqslant\frac{\delta}{p}\lVert\nabla u\rVert_{p}^p-\frac{1}{q}C_{N,p,q}^qa^{q-\frac{N(q-p)}{p}}\lVert\nabla u\rVert_{p}^{\frac{N(q-p)}{p}}.\]
    \end{lemma}
    \begin{proof}
    	By the Gagliardo-Nirenberg inequality, H\"{o}lder inequality and Sobolev inequality, we have
    	    \[\lVert u\rVert_{q}^q\leqslant C_{N,p,q}^q\lVert u\rVert_{p}^{q-\frac{N(q-p)}{p}}\lVert\nabla u\rVert_{p}^{\frac{N(q-p)}{p}}=C_{N,p,q}^qa^{q-\frac{N(q-p)}{p}}\lVert\nabla u\rVert_{p}^{\frac{N(q-p)}{p}},\]
    	and
    	    \[\int_{\mathbb{R}^N} V(x)\lvert u\rvert^pdx\leqslant\lVert V\rVert_{N/p}\lVert u\rVert_{p^*}^p\leqslant S^{-1}\lVert V\rVert_{N/p}\lVert\nabla u\rVert_{p}^p.\]
    	Hence,
    	    \[E(u)\geqslant\frac{1}{p}(1-S^{-1}\lVert V\rVert_{N/p})\lVert\nabla u\rVert_{p}^p-\frac{1}{q}C_{N,p,q}^qa^{q-\frac{N(q-p)}{p}}\lVert\nabla u\rVert_{p}^{\frac{N(q-p)}{p}},\]
    	which together with (\ref{Vcon1}) implies
    	    \[E(u)\geqslant\frac{\delta}{p}\lVert\nabla u\rVert_{p}^p-\frac{1}{q}C_{N,p,q}^qa^{q-\frac{N(q-p)}{p}}\lVert\nabla u\rVert_{p}^{\frac{N(q-p)}{p}}.\]
    \end{proof}

    \begin{lemma}\label{sconvergence}
    	For every $u\in S_{a}$,
    	    \begin{equation}\label{sstaru}
    	    	\lim_{s\rightarrow-\infty}\lVert\nabla(s\star u)\rVert_{p}^p=0,\quad\lim_{s\rightarrow+\infty}\lVert\nabla(s\star u)\rVert_{p}=+\infty,
    	    \end{equation}
    	and
    	    \begin{equation}\label{Fsstaru}
    	    	\lim_{s\rightarrow-\infty}E(s\star u)=0,\quad\lim_{s\rightarrow+\infty}E(s\star u)=-\infty.
    	    \end{equation}
    \end{lemma}
    \begin{proof}
    	It is obvious to obtain (\ref{sstaru}), here we only prove (\ref{Fsstaru}). By the H\"{o}lder inequality,
    	    \[\int_{\mathbb{R}^N}V(x)\lvert s\star u\rvert^pdx\leqslant\lVert V\rVert_{N/p}\lVert s\star u\rVert_{p^*}^p=e^{ps}\lVert V\rVert_{N/p}\lVert u\rVert_{p^*}^p\rightarrow 0\]
    	as $s\rightarrow-\infty$. Thus
    	    \[\lim_{s\rightarrow-\infty}E(s\star u)=\lim_{s\rightarrow-\infty}E_{\infty}(s\star u)-\frac{1}{p}\lim_{s\rightarrow-\infty}\int_{\mathbb{R}^N}V(x)\lvert s\star u\rvert^pdx=0.\]
    	Since $V\geqslant 0$, we have
    	    \[\lim_{s\rightarrow+\infty}E(s\star u)\leqslant\lim_{s\rightarrow+\infty}E_{\infty}(s\star u)=-\infty.\]
    	This ends of the proof.
    \end{proof}

    For every $c\in\mathbb{R}$ and $R>0$, define $E^c=\{u\in S_{a}: E(u)\leqslant c\}$ and
        \[M_{R}=\inf\{E(u): u\in S_{a}, \lVert\nabla u\rVert_{p}=R\}.\]
    From Lemmas \ref{Fgeq} and \ref{sconvergence}, it is easy to know that $E^0\neq\emptyset$ and there exist $\tilde{R}>R_{0}>0$ such that for all $0<R\leqslant R_{0}$, $0<M_{R}<M_{\tilde{R}}$. Thus, we can construct a min-max structure
        \begin{equation}\label{Gamma}
        	\Gamma=\{\xi\in C([0,1],\mathbb{R}\times S_{a}): \xi(0)\in(0,A_{R_{0}}), \xi(1)\in(0,E^0)\}
        \end{equation}
    with associated min-max level
        \begin{equation}\label{mVa}
        	m_{V,a}=\inf_{\xi\in\Gamma}\max_{t\in[0,1]}\tilde{E}(\xi(t))>0,
        \end{equation}
    where
        \[A_{k}=\{u\in S_{a}: \lVert\nabla u\rVert_{p}\leqslant R\}\quad\mbox{and}\quad\tilde{E}(s,u)=E(s\star u).\]

    Next, we prove $E$ has a bounded PS sequence. Before starting the proof, we give the definition of homotopy-stable family.

    \begin{definition}\label{homo}
    	{\rm(\cite[Definition 5.1]{gn})} Let $B$ be a closed subset of $X$. We shall say that a class of $\mathcal{F}$ of compact subsets of $X$ is a homotopy-stable family with extended boundary $B$ if for any set $A$ in $\mathcal{F}$ and any $\eta\in C([0,1]\times X,X)$ satisfying $\eta(t,x)=x$ for all $(t,x)$ in $(\{0\}\times X)\cup([0,1]\times B)$ we have that $\eta(\{1\}\times A)\in\mathcal{F}$.
    \end{definition}

    \begin{lemma}\label{homops}
    	{\rm(\cite[Theorem 5.2]{gn})} Let $\varphi$ be a $C^1$-functional on a complete connected $C^1$-Finsler manifold $X$ and consider a homotopy-stable family $\mathcal{F}$ with an extend closed boundary $B$. Set $c=c(\varphi,\mathcal{F})$ and let $F$ be a closed subset of $X$ satisfying

        {\rm(F'1)} $A\cap F\backslash B\neq\emptyset$ for every $A\in\mathcal{F}$,\\
        and

        {\rm(F'2)} $\sup\varphi(B)\leqslant c\leqslant\inf\varphi(F)$.\\
        Then, for any sequence of sets $\{A_{n}\}$ in $\mathcal{F}$ such that $\lim_{n\rightarrow\infty}\sup_{A_{n}}\varphi=c$, there exists a sequence $\{x_{n}\}$ in $X\backslash B$ such that

        {\rm(i)} $\lim_{n\rightarrow\infty}\varphi(x_{n})=c$,

        {\rm(ii)} $\lim_{n\rightarrow\infty}\lVert d\varphi(x_{n})\rVert=0$,

        {\rm(iii)} $\lim_{n\rightarrow\infty}\mbox{\rm dist}(x_{n},F)=0$,

        {\rm(iv)} $\lim_{n\rightarrow\infty}\mbox{\rm dist}(x_{n},A_{n})=0$.
    \end{lemma}

    \begin{lemma}\label{PS}
    	There exists a bounded PS sequence $\{u_{n}\}$ for $E|_{S_{a}}$ at the level $m_{V,a}$, that is
    	    \begin{equation}\label{unps}
    	    	E(u_{n})\rightarrow m_{V,a}\quad\mbox{and}\quad\lVert E'(u_{n})\rVert_{(T_{u_{n}}S_{a})^*}\rightarrow 0\quad\mbox{as}\ n\rightarrow\infty,
    	    \end{equation}
        such that
            \begin{equation}\label{Phozaev}
            	\lVert\nabla u_{n}\rVert_{p}^p-\frac{N(q-p)}{pq}\lVert u_{n}\rVert_{q}^q-\frac{1}{p}\int_{\mathbb{R}^N}V(x)(N\lvert u_{n}\rvert^p+p\lvert u_{n}\rvert^{p-2}u_{n}\nabla u_{n}\cdot x)dx\rightarrow 0
            \end{equation}
        as $n\rightarrow\infty$. Moreover, the related Lagrange multipliers of $\{u_{n}\}$
            \begin{equation}\label{lamn}
            	\lambda_{n}=-\frac{1}{a^p}F'(u_{n})u_{n},
            \end{equation}
         converges to $\lambda>0$ as $n\rightarrow\infty$.
    \end{lemma}
    \begin{proof}
    	Firstly, we prove the existence of PS sequence. Let $X=\mathbb{R}\times S_{a}$, $\mathcal{F}=\Gamma$ which is given by (\ref{Gamma}) and $B=(0,A_{R_{0}})\cup(0,E^0)$. By Definition \ref{homo}, $\Gamma$ is a homotopy-stable family of compact subsets of $\mathbb{R}\times S_{a}$ with extend closed boundary $(0,A_{R_{0}})\cup(0,E^0)$. Let
    	    \[\varphi=\tilde{E}(s,u)\quad c=m_{V,a}\quad\mbox{and}\quad F=\{(s,u)\in\mathbb{R}\times S_{a}:\tilde{E}(s,u)\geqslant c\}.\]
    	We can check that $E$ satisfies (F'1) and (F'2) in Lemma \ref{homops}. Considering the sequence
    	    \[\{A_{n}\}:=\{\xi_{n}\}=\{(\alpha_{n},\beta_{n})\}\subset\Gamma\]
    	such that
    	    \[m_{V,a}\leqslant\max_{t\in[0,1]}\tilde{F}(\xi_{n}(t))<m_{V,a}+\frac{1}{n}\]
    	(we may assume that $\alpha_{n}=0$, if not, replacing $\{(\alpha_{n},\beta_{n})\}$ with $\{(0,\alpha_{n}\star\beta_{n})\}$). Then, by Lemma \ref{homops}, there exists a sequence $\{(s_{n},v_{n})\}$ for $\tilde{E}|_{S_{a}}$ at level $m_{V,a}$ such that
    	    \[\partial_{s}\tilde{F}(s_{n},v_{n})\rightarrow 0\quad\mbox{and}\quad\lVert\partial_{u}\tilde{F}(s_{n},v_{n})\rVert_{(T_{v_{n}}S_{a})^*}\rightarrow 0\quad\mbox{as}\ n\rightarrow\infty.\]
    	Moreover, by (iv) in Lemma \ref{homops}, we have
    	    \[\lvert s_{n}\rvert+\mbox{dist}_{W^{1,p}}(v_{n},\beta_{n}([0,1]))\rightarrow 0\quad\mbox{as}\ n\rightarrow\infty,\]
    	which implies $s_{n}\rightarrow 0$ as $n\rightarrow 0$.
    	Therefore, we have
    	    \[E(s_{n}\star v_{n})=\tilde{E}(s_{n},v_{n})\rightarrow m_{V,a}\quad\mbox{as}\ n\rightarrow\infty\]
    	and
    	    \[E'(s_{n}\star v_{n})(s_{n}\star\varphi)=\partial_{u}\tilde{E}(s_{n},v_{n})\varphi=o_{n}(1)\lVert\varphi\rVert=o_{n}(1)\lVert s_{n}\star\varphi\rVert\]
    	for every $\varphi\in T_{v_{n}}S_{a}$. Let $\{u_{n}\}=\{(s_{n}\star v_{n})\}$. Then $\{u_{n}\}$ is a PS sequence for $E|_{S_{a}}$ at the level $m_{V,a}$. Differentiating $\tilde{E}$ with respect to $s$, we obtain (\ref{Phozaev}).
    	
    	Next, we claim that $\{u_{n}\}$ is bounded in $W^{1,p}(\mathbb{R}^N)$. Set
    	    \[a_{n}:=\lVert\nabla u_{n}\rVert_{p}^p,\quad b_{n}:=\lVert u_{n}\rVert_{q}^q,\]
    	    \[c_{n}:=\int_{\mathbb{R}^N}V(x)\lvert u_{n}\rvert^pdx\quad\mbox{and}\quad d_{n}:=\int_{\mathbb{R}^N}V(x)\lvert u_{n}\rvert^{p-2}u_{n}\nabla u_{n}\cdot xdx.\]
    	By (\ref{unps}) and (\ref{Phozaev}), we have
    	    \begin{equation}\label{anbncnm}
    	    	\frac{1}{p}a_{n}-\frac{1}{q}b_{n}-\frac{1}{p}c_{n}=m_{V,a}+o_{n}(1)
    	    \end{equation}
    	and
    	    \begin{equation}\label{anbncndn}
    	    	a_{n}-\frac{N(q-p)}{pq}b_{n}-\frac{N}{p}c_{n}-d_{n}=o_{n}(1),
    	    \end{equation}
    	which implies
    	    \[a_{n}=\frac{Np(q-p)}{B}m_{V,a}-\frac{N(2p-q)}{B}c_{n}-\frac{p^2}{B}d_{n}+o_{n}(1).\]
    	Since
    	    \begin{equation}\label{cndn}
    	    	\lvert c_{n}\rvert\leqslant S^{-1}\lVert V\rVert_{N/p}a_{n},\quad\mbox{and}\quad\lvert d_{n}\rvert\leqslant S^{-\frac{p-1}{p}}\lVert W\rVert_{N/(p-1)}a_{n},
    	    \end{equation}
    	we get
    	    \[\Big(B-N\lvert q-2p\rvert S^{-1}\lVert V\rVert_{N/p}-p^2S^{-\frac{p-1}{p}}\lVert W\rVert_{N/(p-1)}\Big)a_{n}=Np(q-p)m_{V,a}+o_{n}(1).\]
    	By assumption (\ref{Vcon2}), $\{a_{n}\}$ is bounded and
    	    \begin{equation}\label{an}
    	    	a_{n}\leqslant\frac{Np(q-p)m_{V,a}}{B-N\lvert q-2p\rvert S^{-1}\lVert V\rVert_{N/p}-p^2S^{-\frac{p-1}{p}}\lVert W\rVert_{N/(p-1)}}+o_{n}(1).
    	    \end{equation}
    	Finally, we claim $\lambda>0$. By (\ref{unps}), there exists $\lambda_{n}>0$ such that
    	    \[E'(u_{n})\varphi=\lambda_{n}\int_{\mathbb{R}^N}\lvert u_{n}\rvert^{p-2}u_{n}\varphi dx+o_{n}(1)\lVert\varphi\rVert_{W^{1,p}},\]
    	and hence we can choose $\lambda_{n}$ as (\ref{lamn}). Since $a_{n}$ is bounded, so $b_{n}$, $c_{n}$, $d_{n}$ and $\lambda_{n}$ are also bounded, we assume that they converge to $a_{0}$, $b_{0}$, $c_{0}$, $d_{0}$ and $\lambda$ respectively. By (\ref{anbncnm}), (\ref{anbncndn}), (\ref{cndn}) and (\ref{an}), we have
    	    \begin{align*}
    	    	\lambda a^p&=-\lim_{n\rightarrow\infty}E'(u_{n})u_{n}=-a_{0}+b_{0}+c_{0}\\
    	    	&=\frac{1}{B}\Big\{p[Np-(N-p)q]m_{V,a}-(N-p)(q-p)c-p(q-p)d_{0}\Big\}\\
    	    	&\geqslant\frac{1}{B}\Big\{p[Np-(N-p)q]m_{V,a}-(N-p)(q-p)S^{-1}\lVert V\rVert_{N/p}a_{0}- p(q-p)S^{-\frac{p-1}{p}}\lVert W\rVert_{N/(p-1)}a_{0}\Big\}\\
    	    	&\geqslant\frac{p}{B}\Big\{A-\frac{(N-p)DS^{-1}\lVert V\rVert_{N/p}}{B-N\lvert q-2p\rvert S^{-1}\lVert V\rVert_{N/p}-p^2S^{-(p-1)/p}\lVert W\rVert_{N/(p-1)}}-\\
    	    	&\qquad\qquad\frac{pDS^{-(p-1)/p}\lVert W\rVert_{N/(p-1)}}{B-N\lvert q-2p\rvert S^{-1}\lVert V\rVert_{N/p}-p^2S^{-(p-1)/p}\lVert W\rVert_{N/(p-1)}}\Big\}m_{V,a}\\
    	    	&=\frac{p}{B}\frac{AB-[AN\lvert q-2p\rvert+(N-p)D]S^{-1}\lVert V\rVert_{N/p}-p(pA+D)S^{-(p-1)/p}\lVert W\rVert_{N/(p-1)}}{B-N\lvert q-2p\rvert S^{-1}\lVert V\rVert_{N/p}-p^2S^{-(p-1)/p}\lVert W\rVert_{N/(p-1)}}m_{V,a}.
    	    \end{align*}
        Therefore, assumption (\ref{Vcon3}) implies $\lambda>0$.
    \end{proof}

    \begin{lemma}\label{Fgeq0}
    	Let $u$ be a weak solution of {\rm(\ref{equation})}. If {\rm(\ref{Vcon4})} holds, then $E(u)\geqslant 0$.
    \end{lemma}
    \begin{proof}
    	By the Pohozaev identity, assumption (\ref{Vcon4}) and (\ref{cndn}), we have
    	    \begin{align*}
    	    	E(u)&=\frac{1}{p}\lVert\nabla u\rVert_{p}^p-\frac{1}{q}\lVert u\rVert_{q}^q-\frac{1}{p}\int_{\mathbb{R}^N}V(x)\lvert u\rvert^{p-2}dx\\
    	    	&=\frac{N(q-p)-p^2}{Np(q-p)}\lVert\nabla u\rVert_{p}^p+\frac{2p-q}{p(q-p)}\int_{\mathbb{R}^N}V(x)\lvert u\rvert^pdx+\frac{p}{N(q-p)}\int_{\mathbb{R}^N}V(x)\lvert u\rvert^{p-2}u\nabla u\cdot xdx.\\
    	    	&\geqslant\frac{1}{Np(q-p)}\Big[N(q-p)-p^2-N\lvert q-2p\rvert S^{-1}\lVert V\rVert_{N/p}-p^2S^{-\frac{p-1}{p}}\lVert W\rVert_{N/(p-1)}\Big]\lVert\nabla u\rVert_{p}^p\\
    	    	&\geqslant 0.
    	    \end{align*}
    \end{proof}

    \begin{lemma}\label{mvama}
    	For every $a>0$, there holds $m_{V,a}<m_{a}$.
    \end{lemma}
    \begin{proof}
    	By Lemma \ref{mageqz}, there exists $w_{a}\in Z_{a}$ such that $E_{\infty}(w_{a})=m_{a}$. Let
    	    \[\xi(t)=\big(0,[(1-t)h_{0}+th_{1}]\star w_{a}\big),\]
    	where $h_{0}<<-1$ such that $\lVert\nabla(h_{0}\star w_{a})\rVert_{p}<R_{0}$ and $h_{1}>>1$ such that $E(h_{1}\star w_{a})<0$. Then, $\xi\in\Gamma$ and $\max_{t\in[0,1]}\tilde{E}(\xi(t))\geqslant m_{V,a}$. Since
    	    \begin{align*}
    	    	\max_{t\in[0,1]}\tilde{E}(\xi(t))&=\max_{t\in[0,1]}F\big([(1-t)h_{0}+th_{1}]\star w_{a}\big)\\
    	    	&<\max_{t\in[0,1]}E_{\infty}\big([(1-t)h_{0}+th_{1}]\star w_{a}\big)\leqslant\max_{s\in\mathbb{R}}E_{\infty}(s\star w_{a})\\
    	    	&=m_{a},
    	    \end{align*}
        we have $m_{V,a}<m_{a}$.
    \end{proof}

    \noindent\textbf{End of the proof of Theorem \ref{th1}.} Lemma \ref{Fgeq0} shows that (\ref{equation}) has no solution with negative energy. Now, let us prove the existence result. By Lemma \ref{PS}, there exist a PS sequence $\{u_{n}\}$ for $E|_{S_{a}}$ at level $m_{V,a}$ and $u\in W^{1,p}(\mathbb{R}^N)$ such that $u_{n}\rightharpoonup u$ in $W^{1,p}(\mathbb{R}^N)$. Moreover, the Lagrange multipliers $\lambda_{n}\rightarrow\lambda>0$ as $n\rightarrow\infty$.

    Next, we prove $u_{n}\rightarrow u$ in $W^{1,p}(\mathbb{R}^N)$. It is not difficult to prove that $\{u_{n}\}$ is a PS sequence for $E_{\lambda}$ at level $m_{V,a}+\frac{\lambda}{p}a^p$. Suppose by contradiction that $\{u_{n}\}$ does not strongly convergence to $u$ in $W^{1,p}(\mathbb{R}^N)$. Then, by Lemma \ref{splittig}, there exists $k\in\mathbb{N}_{+}$ and $y_{n}^j\in\mathbb{R}^N(1\leqslant j\leqslant k)$ such that
        \[u_{n}=u+\sum_{j=1}^{k}w^j(\cdot-y_{n}^j)+o_{n}(1)\quad\mbox{in}\ W^{1,p}(\mathbb{R}^N),\]
    where $w^j$ satisfies
        \[-\Delta_{p}w+\lambda\lvert w\rvert^{p-2}w=\lvert w\rvert^{q-2}w.\]
    Let $n\rightarrow\infty$, by (\ref{decomener}), we have
        \[m_{V,a}+\frac{\lambda}{p}a^p=E_{\lambda}(u)+\sum_{j=1}^{k}F_{\infty,\lambda}(w^j)=E(u)+\sum_{j=1}^{k}E_{\infty}(w^j)+\frac{\lambda}{p}\lVert u\rVert_{p}^p+\frac{\lambda}{p}\sum_{j=1}^{k}\lVert w^j\rVert_{p}^p.\]
    By (\ref{decomnorm}), we know
        \[a^p=\lVert u_{n}\rVert_{p}^p=\lVert u\rVert_{p}^p+\sum_{j=1}^{k}\lVert w^j\rVert_{p}^p.\]
    Thus,
        \begin{equation}\label{mVaequEu+Einfwj}
        	m_{V,a}=E(u)+\sum_{j=1}^{k}E_{\infty}(w^j).
        \end{equation}
    Since $\lVert w^j\rVert_{p}^p\leqslant a^p$, by Lemma \ref{made},
        \[E_{\infty}(w^j)\geqslant m_{\lVert w^j\rVert_{p}}\geqslant m_{a},\]
    which together with (\ref{mVaequEu+Einfwj}) and Lemma \ref{Fgeq0} implies $m_{V,a}\geqslant m_{a}$, a contradiction with Lemma \ref{mvama}.

    Now, by strong convergence, we know $u\in S_{a}$ satisfies (\ref{equation})$.\hfill\Box$

\section{Proof of Theorem \ref{th2}}
\subsection{Existence of local minimizer}
    We give following assumption on $V$ and $a$.
        \begin{equation}\label{Vcon4min}
        	a^{\gamma}\lVert V\rVert_{r}^{\frac{r[N(q-p)-p^2]}{N(q-p)r-Np}}<K,
        \end{equation}
    where $K>0$ depending on $N,p,q$ and $r$, and
        \begin{equation*}
        	\gamma=q-\frac{N(q-p)}{p}+\frac{[(N-p)(q-p)r-Np][N(q-p)-p^2]}{p[N(q-p)r-Np]}.
        \end{equation*}
    We will see later that $\gamma>0$, hence (\ref{Vcon2min}) can be established.

    Let
        \[\alpha=\frac{pr}{r-1}\Longleftrightarrow r=\frac{\alpha}{\alpha-p}.\]
    For every $u\in S_{a}$, by the Gagliardo-Nirenberg inequality and H\"{o}lder inequality, we have
        \[\lVert u\rVert_{q}^q\leqslant C_{N,p,q}^q\lVert u\rVert_{p}^{q-\frac{N(q-p)}{p}}\lVert\nabla u\rVert_{p}^{\frac{N(q-p)}{p}}=C_{N,p,q}^qa^{q-\frac{N(q-p)}{p}}\lVert\nabla u\rVert_{p}^{\frac{N(q-p)}{p}},\]
    and
        \[\int_{\mathbb{R}^N}V(x)\lvert u\rvert^pdx\leqslant\lVert V\rVert_{r}\lVert u\rVert_{\alpha}^p\leqslant C_{N,p,\alpha}^p\lVert V\rVert_{r}a^{p-\frac{N}{r}}\lVert\nabla u\rVert_{p}^{\frac{N}{r}},\]
    which implies
        \begin{equation}\label{FgeqVr}
        	E(u)\geqslant\frac{1}{p}\lVert\nabla u\rVert_{p}^p-\frac{1}{q}C_{N,p,q}^qa^{q-\frac{N(q-p)}{p}}\lVert\nabla u\rVert_{p}^{\frac{N(q-p)}{p}}-\frac{1}{p}C_{N,p,\alpha}^p\lVert V\rVert_{r}a^{p-\frac{N}{r}}\lVert\nabla u\rVert_{p}^{\frac{N}{r}}.
        \end{equation}
    Define the function $h:\mathbb{R}^+\rightarrow\mathbb{R}$
        \begin{equation}\label{h}
        	h(t):=\frac{1}{p}t^p-\frac{1}{q}C_{N,p,q}^qa^{q-\frac{N(q-p)}{p}}t^{\frac{N(q-p)}{p}}-\frac{1}{p}C_{N,p,\alpha}^p\lVert V\rVert_{r}a^{p-\frac{N}{r}}t^{\frac{N}{r}}.
        \end{equation}
    By (\ref{FgeqVr}), there is $E(u)\geqslant h(\lVert\nabla u\rVert_{p})$ for every $u\in S_{a}$.

    \begin{lemma}\label{infgeq0}
    	Let $N\geqslant 1$, $p>1$, $r\in(\max\{1,\frac{N}{p}\},+\infty]$ and assumption {\rm(\ref{Vcon4min})} holds. Then, for every $a>0$, there exists positive constants $C$ depending on $N$, $p$, $q$ and $r$ such that
    	    \[\inf\{E(u): u\in S_{a}, \lVert\nabla u\rVert_{p}=R_{a}\}>0,\]
    	where
    	    \[R_{a}=Ca^{\frac{(N-p)(q-p)r-Np}{N(q-p)r-Np}}\lVert V\rVert_{r}^{\frac{pr}{N(q-p)r-Np}}.\]
    \end{lemma}
    \begin{proof}
    	$h(t)>0$ if and only if $g(t)<1$ with
            \[g(t)=\frac{p}{q}C_{N,p,q}^qa^{q-\frac{N(q-p)}{p}}t^{\frac{N(q-p)}{p}-p}+C_{N,p,\alpha}^p\lVert V\rVert_{r}a^{p-\frac{N}{r}}t^{\frac{N}{r}-p}.\]
        Since $r>N/p$, it is not difficult to know that $g$ has a unique critical point
            \[\bar{t}=Ca^{\frac{(N-p)(q-p)r-Np}{N(q-p)r-Np}}\lVert V\rVert_{r}^{\frac{pr}{N(q-p)r-Np}}\]
        which is a global minimizer, where $C>0$ depending on $N$, $p$, $q$ and $r$. The minimum level is
            \[g(\bar{t})=Ca^{\gamma}\lVert V\rVert_{r}^{\frac{r[N(q-p)-p^2]}{N(q-p)r-Np}},\]
        where
            \[\gamma=q-\frac{N(q-p)}{p}+\frac{[(N-p)(q-p)r-Np][N(q-p)-p^2]}{p[N(q-p)r-Np]}>0,\]
        since the minimum level is increasing with respect to $a$. Therefore, if there exists $t>0$ such that $g(t)<1$, we must have
            \[a^{\gamma}\lVert V\rVert_{r}^{\frac{r[N(q-p)-p^2]}{N(q-p)r-Np}}<K,\]
        where $K>0$ depends on $N$, $p$, $q$ and $r$, that is (\ref{Vcon4min}).

        Let $R_{a}=\bar{t}$. Then, under assumption (\ref{Vcon4min}), we have $h(R_{a})>0$. Thus, by (\ref{FgeqVr}), there is
            \[\inf\{E(u): u\in S_{a}, \lVert\nabla u\rVert_{p}=R_{a}\}\geqslant h(\lVert\nabla u\rVert_{p})=h(R_{a})>0.\]
    \end{proof}

    Let $a_{*}$ be the supremum of $a$ that makes $(\ref{Vcon4min})$ holds. Define
        \begin{equation}\label{cVa}
        	c_{V,a}=\inf\{E(u): u\in S_{a}, \lVert\nabla u\rVert_{p}\leqslant R_{a}\},
        \end{equation}
    where $0<a<a_{*}$ and $R_{a}$ is given by Lemma \ref{infgeq0}.

    \begin{lemma}\label{cValeq0}
    	If assumption {\rm(\ref{Vcon3min})} holds, then $c_{V,a}<0$ for every $0<a<a_{*}$.
    \end{lemma}
    \begin{proof}
    	Let $\varphi\in S_{a}$ such that (\ref{Vcon3min}) holds. For every $t>0$, we have
    	    \[E(t\varphi)\leqslant-\frac{1}{q}t^q\lVert\varphi\rVert_{q}^q<0.\]
    	
    	Let $\bar{t}=\frac{R_{a}}{\lVert\nabla\varphi\rVert_{p}}$. Then $\lVert\nabla(\bar{t}\varphi)\rVert_{p}=R_{a}$ and $E(\bar{t}\varphi)<0$. It is not difficult to prove that for every $0<b\leqslant a$, there is
    	    \begin{equation}\label{uSbnabequaRa}
    	    	\inf\{E(u): u\in S_{b}, \lVert\nabla u\rVert_{p}=R_{a}\}>0.
    	    \end{equation}
    	If $\bar{t}\leqslant 1$. Since
    	    \[\bar{t}\varphi\in\{E(u): u\in S_{\bar{t}a}, \lVert\nabla u\rVert_{p}=R_{a}\},\]
    	by (\ref{uSbnabequaRa}), $E(\bar{t}\varphi)<0$, a contradiction. Therefore, $\bar{t}>1$ and
    	    \[\varphi\in\{u\in S_{a}: \lVert\nabla u\rVert_{p}\leqslant R_{a}\}.\]
    	By the definition of $c_{V,a}$, we have
    	    \[c_{V,a}\leqslant E(\varphi)<0.\]
    \end{proof}

    \begin{lemma}\label{cVadec}
    	If assumption {\rm(\ref{Vcon3min})} holds, then for every $0<b<a<a_{*}$, there is
    	    \[c_{V,a}\leqslant\inf\{E(u): u\in S_{b}, \lVert\nabla u\rVert_{p}\leqslant R_{a}\}<0.\]
    \end{lemma}
    \begin{proof}
    	Let $\varphi\in S_{a}$ such that (\ref{Vcon3min}) holds. By Lemma \ref{cValeq0}, we know
    	    \[\varphi\in\{u\in S_{a}: \lVert\nabla u\rVert_{p}\leqslant R_{a}\}.\]
    	Since
    	    \[\psi=\frac{b}{a}\varphi\in\{u\in S_{b}: \lVert\nabla u\rVert_{p}\leqslant R_{a}\}\]
    	and
    	    \[E(\psi)\leqslant-\frac{1}{q}\Big(\frac{b}{a}\Big)^q\lVert\varphi\rVert_{q}^q<0,\]
    	we have
    	    \[\inf\{E(u): u\in S_{b}, \lVert\nabla u\rVert_{p}\leqslant R_{a}\}<0.\]
    	
    	For sufficiently small $\varepsilon>0$, there exists $u\in S_{b}$ satisfies $\lVert\nabla u\rVert_{p}\leqslant R_{a}$ such that
    	    \[E(u)<\inf\{E(u): u\in S_{b}, \lVert\nabla u\rVert_{p}\leqslant R_{a}\}+\varepsilon<0.\]
    	Let $v=\frac{a}{b}u\in S_{a}$. Then,
    	    \begin{align*}
    	    	E(v)&=\frac{1}{p}\Big(\frac{a}{b}\Big)^p\Big(\lVert\nabla u\rVert_{p}^p-\int_{\mathbb{R}^N}V(x)\lvert u\rvert^pdx\Big)-\frac{1}{q}\Big(\frac{a}{b}\Big)^q\lVert u\rVert_{q}^q\leqslant\Big(\frac{a}{b}\Big)^pE(u)\\
    	    	&<\inf\{E(u): u\in S_{b}, \lVert\nabla u\rVert_{p}\leqslant R_{a}\}+\varepsilon.
    	    \end{align*}
    	Now, we claim that $\lVert\nabla v\rVert_{p}\leqslant R_{a}$. If not, there exists $1\leqslant\bar{t}<\frac{a}{b}$ such that $\lVert\nabla(\bar{t}u)\rVert_{p}=R_{a}$, $\lVert\bar{t}u\rVert_{p}<a$ and
    	    \[E(\bar{t}u)\leqslant\bar{t}^pE(u)<0,\]
    	a contradiction with (\ref{uSbnabequaRa}). Therefore,
    	    \[v\in\{u\in S_{a}: \lVert\nabla u\rVert_{p}\leqslant R_{a}\}.\]
    	By the definition of $c_{V,a}$, we have
    	    \[c_{V,a}\leqslant E(v)<\inf\{E(u): u\in S_{b}, \lVert\nabla u\rVert_{p}\leqslant R_{a}\}+\varepsilon.\]
    \end{proof}

    \noindent\textbf{Proof of the existence of local minimizer.} Let $\{u_{n}\}$ be a minimizing sequence for $c_{V,a}$. By Lemma \ref{infgeq0}, we know
        \[\liminf_{n\rightarrow\infty}\lVert\nabla u_{n}\rVert_{p}<R_{a}.\]
    Therefore, the Ekeland's variational principle implies $\{u_{n}\}$ can be chosen as a PS sequence for $E|_{S_{a}}$ at level $c_{V,a}$. Since $\{u_{n}\}$ is bounded in $W^{1,p}(\mathbb{R}^N)$, there exists $u\in W^{1,p}(\mathbb{R}^N)$ such that $u_{n}\rightharpoonup u$ in $W^{1,p}(\mathbb{R}^N)$ as $n\rightarrow\infty$. Therefore, the Lagrange multipliers
        \[\lambda_{n}=-\frac{F'(u_{n})u_{n}}{a^p}\]
    are also bounded and there exists $\lambda\in\mathbb{R}$ such that $\lambda_{n}\rightarrow\lambda$ as $n\rightarrow\infty$. Since $\{u_{n}\}$ is a PS sequence for $E|_{S_{a}}$, we have
        \[F'(u_{n})u_{n}+\lambda_{n}a^p=o_{n}(1),\]
    that is
        \[o_{n}(1)=\frac{1}{p}\Big(\lVert\nabla u_{n}\rVert_{p}^p-\lVert u_{n}\rVert_{q}^q-\int_{\mathbb{R}^N}V(x)\lVert u_{n}\rVert_{p}^p\Big)+\frac{1}{p}\lambda_{n}a^p\leqslant F(u_{n})+\frac{1}{p}\lambda_{n}a^p=c_{V,a}+\frac{1}{p}\lambda a^p+o_{n}(1),\]
    which implies
        \[\lambda\geqslant-\frac{pc_{V,a}}{a^p}>0.\]

    If $\{u_{n}\}$ does not strongly convergence to $u$ in $W^{1,p}(\mathbb{R}^N)$. It is not difficult to prove that $\{u_{n}\}$ is also a PS sequence for $E_{\lambda}$ at level $c_{V,a}+\frac{\lambda}{p}a^p$. Then, by Lemma \ref{splittig}, there exists $k\in\mathbb{N}_{+}$ and $y_{n}^j\in\mathbb{R}^N(1\leqslant j\leqslant k)$ such that
        \[u_{n}=u+\sum_{j=1}^{k}w^j(\cdot-y_{n}^j)+o_{n}(1)\quad\mbox{in}\ W^{1,p}(\mathbb{R}^N),\]
    where $w^j$ satisfies
        \[-\Delta_{p}w+\lambda\lvert w\rvert^{p-2}w=\lvert w\rvert^{q-2}w.\]
    Let $n\rightarrow\infty$, by (\ref{decomener}), we have
        \[c_{V,a}+\frac{\lambda}{p}a^p=E_{\lambda}(u)+\sum_{j=1}^{k}F_{\infty,\lambda}(w^j)=E(u)+\sum_{j=1}^{k}E_{\infty}(w^j)+\frac{\lambda}{p}\lVert u\rVert_{p}^p+\frac{\lambda}{p}\sum_{j=1}^{k}\lVert w^j\rVert_{p}^p.\]
    By (\ref{decomnorm}), we know
        \[a^p=\lVert u_{n}\rVert_{p}^p=\lVert u\rVert_{p}^p+\sum_{j=1}^{k}\lVert w^j\rVert_{p}^p.\]
    Thus,
        \begin{equation}\label{cVaFu}
        	c_{V,a}=E(u)+\sum_{j=1}^{k}E_{\infty}(w^j).
        \end{equation}
    Set $\lVert u\rVert_{p}=b\leqslant a$. Since
        \[\lVert\nabla u\rVert_{p}\leqslant\liminf_{n\rightarrow\infty}\lVert\nabla u_{n}\rVert_{p}<R_{a},\]
    by Lemma \ref{cVadec}, we have
        \[E(u)\geqslant\inf\{F(v): v\in S_{b}, \lVert\nabla v\rVert_{p}\leqslant R_{a}\}\geqslant c_{V,a},\]
    which together with (\ref{cVaFu}) implies
        \[\sum_{j=1}^{k}E_{\infty}(w^j)\leqslant 0,\]
    a contradiction with Lemma \ref{mageqz}.

    Now, by strong convergence, we know $u\in S_{a}$ satisfies (\ref{equation})$\hfill\Box$

    \begin{lemma}\label{existphi}
    	Under the assumption of {\rm(\ref{Vinf})}, {\rm(\ref{Vcon3min})} holds.
    \end{lemma}
    \begin{proof}
    	Without loss of generality, we can assume that $x_{0}=0$. If $N<p$, consider the function $\varphi(x)=A_{\varepsilon}e^{-\varepsilon\lvert x\rvert}\in S_{a}$, then
    	    \begin{align*}
    	    	\int_{\mathbb{R}^N}(\lvert\nabla\varphi\rvert^p-V(x)\lvert\varphi\rvert^p)dx&\leqslant\int_{\mathbb{R}^N}\lvert\nabla\varphi\rvert^pdx-\eta\int_{B_{\delta}}\lvert\varphi\rvert^pdx\\
    	    	&=N\omega_{N}\Big(\int_{0}^{+\infty}r^{N-1}\lvert\varphi'(r)\rvert^pdr-\eta\int_{0}^{\delta}r^{N-1}\lvert\varphi(r)\rvert^pdr\Big)\\
    	    	&=N\omega_{N}A_{\varepsilon}^p\Big(\varepsilon^{p-N}\int_{0}^{+\infty}s^{N-1}e^{-ps}ds-\eta\int_{0}^{\delta}r^{N-1}e^{-p\varepsilon r}dr\Big)\\
    	    	&<0
    	    \end{align*}
    	by taking $\varepsilon$ sufficiently small. If $N=p$, consider the function
    	    \begin{equation*}
    	    	\varphi(x)=A_{k}\left\{\begin{array}{ll}
    	    		1,&\lvert x\rvert\leqslant\delta,\\
    	    		\frac{\log(k\delta)-\log\lvert x\rvert}{\log k},&\delta<\lvert x\rvert\leqslant k\delta,\\
    	    		0,&\lvert x\rvert>k\delta,
    	    	\end{array}\right.
    	    \end{equation*}
    	then
    	    \begin{align*}
    	    	\int_{\mathbb{R}^N}(\lvert\nabla\varphi\rvert^N-V(x)\lvert\varphi\rvert^N)dx&\leqslant\int_{\mathbb{R}^N}\lvert\nabla\varphi\rvert^Ndx-\eta\int_{B_{\delta}}\lvert\varphi\rvert^Ndx\\
    	    	&=N\omega_{N}\Big(\int_{\delta}^{+\infty}r^{N-1}\lvert\varphi'(r)\rvert^Ndr-\eta\int_{0}^{\delta}r^{N-1}\lvert\varphi(r)\rvert^Ndr\Big)\\
    	    	&=N\omega_{N}A_{k}^N\Big(\frac{1}{\lvert\log k\rvert^{N-1}}-\frac{\eta\delta^N}{N}\Big)\\
    	    	&<0
    	    \end{align*}
        by taking $k$ sufficiently large. If $N>p$, consider the function
            \begin{equation*}
            	\varphi(x)=A_{k}\left\{\begin{array}{ll}
            		1,&\lvert x\rvert\leqslant\delta,\\
            		(1+\frac{1}{k})\delta^{\frac{N-p}{p-1}}\lvert x\rvert^{\frac{p-N}{p-1}}-\frac{1}{k},&\delta<\lvert x\rvert\leqslant(k+1)^{\frac{p-1}{N-p}}\delta,\\
            		0,&\lvert x\rvert>(k+1)^{\frac{p-1}{N-p}}\delta,
            	\end{array}\right.
            \end{equation*}
        then
             \begin{align*}
            	&\int_{\mathbb{R}^N}(\lvert\nabla\varphi\rvert^p-V(x)\lvert\varphi\rvert^p)dx\\
            	&\leqslant\int_{\mathbb{R}^N}\lvert\nabla\varphi\rvert^pdx-\eta\int_{B_{\delta}}\lvert\varphi\rvert^pdx\\
            	&=N\omega_{N}\Big(\int_{\delta}^{+\infty}r^{N-1}\lvert\varphi'(r)\rvert^pdr-\eta\int_{0}^{\delta}r^{N-1}\lvert\varphi(r)\rvert^pdr\Big)\\
            	&<N\omega_{N}A_{k}^p\Big(\big(\frac{N-p}{p-1}\big)^p\big(1+\frac{1}{k}\big)^p\delta^{\frac{p(N-p)}{p-1}}\int_{\delta}^{+\infty}r^{-\frac{N-1}{p-1}}dr-\frac{\eta\delta^N}{N}\Big)\\
            	&=\omega_{N}A_{k}^p\delta^{N-p}\Big(N\big(\frac{N-p}{p-1}\big)^{p-1}\big(1+\frac{1}{k}\big)^p-\eta\delta^p\Big)<0
            \end{align*}
        by taking $k$ sufficiently large.
    \end{proof}

\subsection{Existence of mountain pass solution}
    We give following assumptions on $V$, $W$ and $a$.
        \begin{equation}\label{Vconmount1}
        	\max\Big\{a^{\big(p-\frac{N}{r}\big)\frac{p(q-p)}{N(q-p)-p^2}}\lVert V\rVert_{r}, a^{\big(p-1-\frac{N}{s}\big)\frac{p(q-p)}{N(q-p)-p^2}}\lVert W\rVert_{s}\Big\}<L_{1},
        \end{equation}
        \begin{equation}\label{Vconmount2}
        	\max\{a^{p-\frac{N}{r}}\lVert V\rVert_{r},a^{p-1-\frac{N}{s}}\lVert W\rVert_{s}\}<L_{2},
        \end{equation}
    where $L_{1},L_{2}$ are positive constants depending on $N,p,q,r$ and $s$. Obviously, (\ref{Vconmount}) can be established under above assumptions.

    \begin{lemma}\label{infgeqkapma}
    	Under the assumption {\rm(\ref{Vconmount1})}, there exist positive constant $C=C(N,p,q)$ and $0<\kappa=\kappa(N,p,q,r)<1$ such that
    	    \[\inf\{E(u): u\in S_{a}, \lVert\nabla u\rVert_{p}=R_{a}\}>\kappa m_{a},\]
    	where
    	    \[R_{a}=Ca^{-\theta}\quad\mbox{with}\quad\theta=\frac{Np-q(N-p)}{N(q-p)-p^2}.\]
    \end{lemma}
    \begin{proof}
    	Considering the function
    	    \[\tilde{h}(t)=\frac{1}{p}t^p-\frac{1}{q}C_{N,p,q}^qa^{q-\frac{N(q-p)}{p}}t^{\frac{N(q-p)}{p}}.\]
    	Then, direct calculations show that
    	    \[\max_{t>0}\tilde{h}(t)=\tilde{h}(\tilde{R})=C_{1}a^{-p\theta},\]
    	where
    	    \[\theta=\frac{Np-q(N-p)}{N(q-p)-p^2},\]
    	$\tilde{R}=Ca^{-\theta}$, $C, C_{1}$ are positive constants depending on $N$, $p$ and $q$. Since
    	    \[h(\tilde{R})=Ca^{-p\theta}(1-C_{2}a^{\tilde{\theta}}\lVert V\rVert_{r}),\]
    	where $h$ is given by (\ref{h}) and
    	    \[\tilde{\theta}=\Big(p-\frac{N}{r}\Big)\frac{p(q-p)}{N(q-p)-p^2}.\]
    	Therefore, under assumption (\ref{Vconmount1}), we can ensure that $h(\tilde{R})\geqslant C_{3}a^{-p\theta}$.
    	
    	Let $R_{a}=\tilde{R}$. Then for every $u\in S_{a}$, $\lVert\nabla u\rVert_{p}=R_{a}$, by (\ref{FgeqVr}), we have
    	    \[E(u)\geqslant h(\lVert\nabla u\rVert_{p})=h(\tilde{R})\geqslant C_{3}a^{-p\theta},\]
    	which implies
    	    \[\inf\{E(u): u\in S_{a}, \lVert\nabla u\rVert_{p}=R_{a}\}>C_{3}a^{-p\theta}.\]
    	By Lemma \ref{made}, we know $m_{a}=C_{4}a^{-p\theta}$. Therefore,
    	    \[\inf\{E(u): u\in S_{a}, \lVert\nabla u\rVert_{p}=R_{a}\}>\kappa m_{a}.\]
    \end{proof}

    By Lemma \ref{infgeqkapma}, we know there exists $0<\tilde{R}_{a}<R_{a}$ such that
        \[\sup\{E(u): u\in S_{a}, \lVert\nabla u\rVert_{p}\leqslant\tilde{R}_{a}\}<\inf\{E(u): u\in S_{a}, \lVert\nabla u\rVert_{p}=R_{a}\}.\]
    Define
       \[c_{a}:=\inf\{E(u): u\in S_{a}, \lVert\nabla u\rVert_{p}\leqslant R_{a}\}.\]
    Now, we can construct a min-max structure
        \begin{equation}\label{Gamma2}
        	\Gamma=\{\xi\in C([0,1],\mathbb{R}\times S_{a}): \xi(0)\in(0,A_{\tilde{R}_{a}}), \xi(1)\in(0,E^{\min\{-1,2c_{a}\}})\},
        \end{equation}
    where
        \[E^c=\{u\in S_{a}: E(u)\leqslant c\}\quad\mbox{and}\quad A_{R}=\{u\in S_{a}: \lVert\nabla u\rVert_{p}\leqslant R\}.\]
    The associated min-max level is
        \[m_{V,a}=\inf_{\xi\in\Gamma}\max_{t\in[0,1]}\tilde{E}(\xi(t))>0,\]
    where $\tilde{E}(s,u)=E(s\star u)$.

    \begin{lemma}\label{mVaboundary}
    	Under the assumption {\rm(\ref{Vconmount1})}, $\kappa m_{a}\leqslant m_{V,a}<m_{a}$.
    \end{lemma}
    \begin{proof}
    	The proof of $m_{V,a}<m_{a}$ is similar to Lemma \ref{mvama}. For every $\xi(t)=(s_{t},u_{t})\in\Gamma$, it is not difficult to know that $\lVert\nabla(s_{0}\star u_{0})\rVert_{p}<R_{a}$ and $\lVert\nabla(s_{1}\star u_{1})\rVert_{p}>R_{a}$. Thus, there exists $\tau\in(0,1)$ such that $\lVert\nabla(s_{\tau}\star u_{\tau})\rVert_{p}=R_{a}$. By Lemma \ref{infgeqkapma}, we have
    	    \[\max_{t\in[0,1]}\tilde{E}(\xi(t))\geqslant\tilde{E}(\xi(\tau))=E(s_{\tau}\star u_{\tau})\geqslant\inf\{E(u): u\in S_{a}, \lVert\nabla u\rVert_{p}=R_{a}\}>\kappa m_{a},\]
    	which implies $m_{V,a}\geqslant\kappa m_{a}$.
    \end{proof}

    Similar to Lemma \ref{PS}, we can obtain a PS sequence $\{u_{n}\}$ for $E|_{S_{a}}$ at level $m_{V,a}$ which satisfies (\ref{Phozaev}).

    \begin{lemma}\label{unbounded}
    	$\{u_{n}\}$ is bounded in $W^{1,p}(\mathbb{R}^N)$.
    \end{lemma}
    \begin{proof}
    	Similar to the proof of Lemma \ref{PS}, we have
    	    \[\frac{1}{p}a_{n}-\frac{1}{q}b_{n}-\frac{1}{p}c_{n}=m_{V,a}+o_{n}(1),\]
    	and
    	    \[a_{n}-\frac{N(q-p)}{pq}b_{n}-\frac{N}{p}c_{n}-d_{n}=o_{n}(1),\]
    	which implies
    	    \[a_{n}=\frac{Np(q-p)}{B}m_{V,a}-\frac{N(2p-q)}{B}c_{n}-\frac{p^2}{B}d_{n}+o_{n}(1).\]
    	By the Gagliardo-Nirenberg inequality and H\"{o}lder inequality, we obtain
    	    \begin{equation}\label{cnleq}
    	    	\lvert c_{n}\rvert\leqslant\lVert V\rVert_{r}\lVert u_{n}\rVert_{\alpha}^p\leqslant C_{N,p,\alpha}^pa^{p-\frac{N}{r}}\lVert V\rVert_{r}\lVert\nabla u\rVert_{p}^{\frac{N}{r}}=C_{N,p,\alpha}^pa^{p-\frac{N}{r}}\lVert V\rVert_{r}a_{n}^{\frac{N}{pr}},
    	    \end{equation}
    	and
    	    \begin{align}\label{dnleq}
    	    	\lvert d_{n}&\rvert\leqslant\lVert W\rVert_{s}\lVert u_{n}\rVert_{\beta}^{p-1}\lVert\nabla u_{n}\rVert_{p}\leqslant C_{N,p,\beta}^{p-1}a^{p-1-\frac{N}{s}}\lVert W\rVert_{s}\lVert\nabla u_{n}\rVert_{p}^{1+\frac{N}{s}}\nonumber\\
    	    	&=C_{N,p,\beta}^{p-1}a^{p-1-\frac{N}{s}}\lVert W\rVert_{s}a_{n}^{\frac{1}{p}\big(1+\frac{N}{s}\big)},
    	    \end{align}
        where
            \[r=\frac{\alpha}{\alpha-p}\quad\mbox{and}\quad s=\frac{p\beta}{(p-1)(\beta-p)}\quad\mbox{with}\ p\leqslant\alpha,\beta<p^*.\]
        Therefore,
            \[a_{n}\leqslant\frac{Np(q-p)}{B}m_{V,a}+\frac{N\lvert 2p-q\rvert}{B}C_{N,p,\alpha}^pa^{p-\frac{N}{r}}\lVert V\rVert_{r}a_{n}^{\frac{N}{pr}}+\frac{p^2}{B}C_{N,p,\beta}^{p-1}a^{p-1-\frac{N}{s}}\lVert W\rVert_{s}a_{n}^{\frac{1}{p}\big(1+\frac{N}{s}\big)}+o_{n}(1).\]
        We know
            \[\frac{N}{pr}<1\quad\mbox{and}\quad\frac{1}{p}\Big(1+\frac{N}{s}\Big)<1.\]
        If $a_{n}\geqslant 1$, by assumption (\ref{Vconmount2}), we can choose $L_{2}$ sufficiently small such that
            \[N\lvert 2p-q\rvert C_{N,p,\alpha}^pa^{p-\frac{N}{r}}\lVert V\rVert_{r}+p^2C_{N,p,\beta}^{p-1}a^{p-1-\frac{N}{s}}\lVert W\rVert_{s}<\frac{B}{2},\]
        which implies
            \[a_{n}\leqslant\frac{2Np(q-p)m_{V,a}}{B}+o_{n}(1).\]
        Therefore, $\{a_{n}\}$ is bounded and
            \begin{equation}\label{anleq}
            	a_{n}\leqslant\max\Big\{1,\frac{2Np(q-p)m_{V,a}}{B}+o_{n}(1)\Big\}.
            \end{equation}
    \end{proof}

    Now, we can deduce that $b_{n}$, $c_{n}$, $d_{n}$ and the Lagrange multipliers $\lambda_{n}$ are bounded as well. Then, we can assume that $a_{n}$, $b_{n}$, $c_{n}$, $d_{n}$ and $\lambda_{n}$ converge to $a_{0}$, $b_{0}$, $c_{0}$, $d_{0}$ and $\lambda$ respectively.

    \begin{lemma}\label{lamgreat0}
    	Under the assumption {\rm(\ref{Vconmount2})}, $\lambda>0$.
    \end{lemma}
    \begin{proof}
    	By (\ref{cnleq}) and (\ref{dnleq}), we have
    	    \begin{align*}
    	    	\lambda a^p&=-\lim_{n\rightarrow\infty}F'(u_{n})u_{n}=-a_{0}+b_{0}+c_{0}\\
    	    	&=\frac{1}{B}\Big\{p[Np-(N-p)q]m_{V,a}-(N-p)(q-p)c_{0}-p(q-p)d_{0}\Big\}\\
    	    	&\geqslant\frac{1}{B}\Big\{p[Np-(N-p)q]m_{V,a}-(N-p)(q-p)C_{N,p,\alpha}^pa^{p-\frac{N}{r}}\lVert V\rVert_{r}a_{0}^{\frac{N}{pr}}\\
    	    	&\qquad\qquad-p(q-p)C_{N,p,\beta}^{p-1}a^{p-1-\frac{N}{s}}\lVert W\rVert_{s}a_{0}^{\frac{1}{p}\big(1+\frac{N}{s}\big)}\Big\}.
    	    \end{align*}
        If $a_{0}\geqslant 1$. Since
            \[\frac{N}{pr}<1\quad\mbox{and}\quad\frac{1}{p}\Big(1+\frac{N}{s}\Big)<1,\]
        then, by (\ref{Vconmount2}) and (\ref{anleq}), for sufficiently small $L_{2}$, there is
            \[\lambda a^p\geqslant C(m_{V,a}-\varepsilon a_{0})\geqslant Cm_{V,a}>0.\]
        If $a_{0}<1$, then, by (\ref{Vconmount2}) and Lemma \ref{mVaboundary}, for sufficiently small $L_{2}$, there is
            \[\lambda a^p\geqslant C(m_{a}-\varepsilon)>0.\]
    \end{proof}

    \noindent\textbf{Proof of the existence of mountain pass type solution.} Firstly, we can prove $\{u_{n}\}$ is a PS sequence for $E_{\lambda}$ at level $m_{V,a}+\frac{\lambda}{p}a^p$. By (\ref{anleq}), we can assume that $u_{n}\rightharpoonup u$ in $W^{1,p}(\mathbb{R}^N)$. If $\{u_{n}\}$ does not strongly convergence to $u$ in $W^{1,p}(\mathbb{R}^N)$, then, by Lemma
    \ref{splittig}, there exists $k\in\mathbb{N}_{+}$ and $y_{n}^j\in\mathbb{R}^N(1\leqslant j\leqslant k)$ such that
        \[u_{n}=u+\sum_{j=1}^{k}w^j(\cdot-y_{n}^j)+o_{n}(1)\quad\mbox{in}\ W^{1,p}(\mathbb{R}^N),\]
    where $w^j$ satisfies
        \[-\Delta_{p}w+\lambda\lvert w\rvert^{p-2}w=\lvert w\rvert^{q-2}w.\]
    Similar to the proof of Theorem \ref{th1}, we have
        \begin{equation}\label{mVaFu}
    	    m_{V,a}=E(u)+\sum_{j=1}^{k}E_{\infty}(w^j).
        \end{equation}
    Denote $\lVert u\rVert_{p}=\mu$ and $\lVert w^j\rVert_{p}=\beta_{j}$. Then
        \[a^p=\mu^p+\sum_{j=1}^{k}\beta_{j}^p.\]
    By Lemma \ref{made} and the definition of $m_{a}$,
        \[\sum_{j=1}^{k}E_{\infty}(w^j)\geqslant\sum_{j=1}^{k}m_{\lVert w^j\rVert_{p}}\geqslant m_{\beta_{1}}=m_{a}\Big(\frac{\beta_{1}}{a}\Big)^{-p\theta},\]
    where
        \[\theta=\frac{Np-q(N-p)}{N(q-p)-p^2}.\]
    We claim that, under the assumption (\ref{Vconmount1}), we have
        \begin{equation}\label{Fugeqma}
        	E(u)\geqslant-\theta m_{a}\Big(\frac{\mu}{a}\Big)^p.
        \end{equation}
    Therefore, using the fact that $\mu^p+\beta_{1}^p\leqslant a^p$, there is
        \begin{align*}
        	E(u)+\sum_{j=1}^{k}E_{\infty}(w^j)&\geqslant-\theta m_{a}\Big(\frac{\mu}{a}\Big)^p+m_{a}\Big(\frac{\beta_{1}}{a}\Big)^{-p\theta}\geqslant m_{a}\Big[-\theta+\theta\Big(\frac{\beta_{1}}{a}\Big)^p+\Big(\frac{\beta_{1}}{a}\Big)^{-p\theta}\Big]\\
        	&\geqslant m_{a}\min_{0<x\leqslant 1}\{-\theta+\theta x+x^{-\theta}\}=m_{a},
        \end{align*}
    a contradiction with $m_{V,a}<m_{a}$.

    Now, by strong convergence, we know $u\in S_{a}$ satisfies (\ref{equation})$\hfill\Box$

    Finally, we prove (\ref{Fugeqma}).
    \begin{lemma}
    	Under the assumption {\rm(\ref{Vconmount1})}, {\rm(\ref{Fugeqma})} is true.
    \end{lemma}
    \begin{proof}
    	Set
    	    \[a_{0}=\lVert\nabla u\rVert_{p}^p,\quad b_{0}=\lVert u\rVert_{q}^q,\quad c_{0}=\int_{\mathbb{R}^N}V(x)\lvert u\rvert^pdx,\quad\mbox{and}\quad d_{0}=\int_{\mathbb{R}^N}V(x)\lvert u\rvert^{p-2}u\nabla u\cdot xdx.\]
    	Then, by Pohozaev identity
    	     \[a_{0}-\frac{N(q-p)}{pq}b_{0}-\frac{N}{p}c_{0}-d_{0}=0,\]
    	(\ref{cnleq}) and (\ref{dnleq}), we have
    	    \begin{align*}
    	    	E(u)&=\frac{1}{p}a_{0}-\frac{1}{q}b_{0}-\frac{1}{p}c_{0}=\frac{N(q-p)-p^2}{Np(q-p)}a_{0}+\frac{2p-q}{p(q-p)}c_{0}+\frac{p}{N(q-p)}d_{0}\\
    	    	&\geqslant C\Big(2a_{0}-C_{1}\mu^{p-\frac{N}{r}}\lVert V\rVert_{r}a_{0}^{\frac{N}{pr}}-C_{1}\mu^{p-1-\frac{N}{s}}\lVert W\rVert_{s}a_{0}^{\frac{1}{p}\big(1+\frac{N}{s}\big)}\Big),
    	    \end{align*}
        where $C, C_{1}$ are positive constants depending on $N$, $p$, $q$, $r$ and $s$. It is not difficult to prove that
            \[a_{0}-C_{1}\mu^{p-\frac{N}{r}}\lVert V\rVert_{r}a_{0}^{\frac{N}{pr}}\geqslant-C_{2}\mu^p\lVert V\rVert_{r}^{\frac{pr}{pr-N}},\]
        and
            \[a_{0}-C_{1}\mu^{p-1-\frac{N}{s}}\lVert W\rVert_{s}a_{0}^{\frac{1}{p}\big(1+\frac{N}{s}\big)}\geqslant-C_{3}\mu^p\lVert W\rVert_{s}^{\frac{ps}{(p-1)s-N}}.\]
        Therefore, under assumption (\ref{Vconmount1}), by Lemma \ref{made}, we have
            \[E(u)\geqslant-\theta m_{a}\Big(\frac{\mu}{a}\Big)^p,\]
        by taking $L_{1}$ sufficiently small.
    \end{proof}

\end{document}